\documentclass[oneside,british,english]{amsart}
\usepackage[T1]{fontenc}
\usepackage[latin9]{inputenc}
\usepackage{mathrsfs}
\usepackage{amsthm}
\usepackage{amsbsy}
\usepackage{amssymb}
\usepackage{setspace}
\PassOptionsToPackage{normalem}{ulem}
\usepackage{ulem}

\makeatletter
\numberwithin{equation}{section}
\numberwithin{figure}{section}
\theoremstyle{plain}
\newtheorem{thm}{\protect\theoremname}
  \theoremstyle{definition}
  \newtheorem{defn}[thm]{\protect\definitionname}
  \theoremstyle{remark}
  \newtheorem{rem}[thm]{\protect\remarkname}
  \theoremstyle{remark}
  \newtheorem{claim}[thm]{\protect\claimname}
  \theoremstyle{plain}
  \newtheorem{lem}[thm]{\protect\lemmaname}
  \theoremstyle{plain}
  \newtheorem{cor}[thm]{\protect\corollaryname}

\theoremstyle{plain}
\newtheorem*{conv}{Convention}

\makeatother

\usepackage{babel}
  \addto\captionsbritish{\renewcommand{\claimname}{Claim}}
  \addto\captionsbritish{\renewcommand{\corollaryname}{Corollary}}
  \addto\captionsbritish{\renewcommand{\definitionname}{Definition}}
  \addto\captionsbritish{\renewcommand{\lemmaname}{Lemma}}
  \addto\captionsbritish{\renewcommand{\remarkname}{Remark}}
  \addto\captionsbritish{\renewcommand{\theoremname}{Theorem}}
  \addto\captionsenglish{\renewcommand{\claimname}{Claim}}
  \addto\captionsenglish{\renewcommand{\corollaryname}{Corollary}}
  \addto\captionsenglish{\renewcommand{\definitionname}{Definition}}
  \addto\captionsenglish{\renewcommand{\lemmaname}{Lemma}}
  \addto\captionsenglish{\renewcommand{\remarkname}{Remark}}
  \addto\captionsenglish{\renewcommand{\theoremname}{Theorem}}
  \providecommand{\claimname}{Claim}
  \providecommand{\corollaryname}{Corollary}
  \providecommand{\definitionname}{Definition}
  \providecommand{\lemmaname}{Lemma}
  \providecommand{\remarkname}{Remark}
\providecommand{\theoremname}{Theorem}

\newcount\skewfactor
\def\mathunderaccent#1#2 {\let\theaccent#1\skewfactor#2
\mathpalette\putaccentunder}
\def\putaccentunder#1#2{\oalign{$#1#2$\crcr\hidewidth
\vbox to.2ex{\hbox{$#1\skew\skewfactor\theaccent{}$}\vss}\hidewidth}}
\def\name{\mathunderaccent\tilde-3 }
\def\Name{\mathunderaccent\widetilde-3 }

\usepackage[hidelinks]{hyperref}
\begin{document}
\makeatletter\def\shfiuwefootnote{\gdef\@thefnmark{}\@footnotetext}\makeatother\shfiuwefootnote{Version 2023-01-15. See \url{https://shelah.logic.at/papers/1085/} for possible updates.}

\title{Generalizing Random Real Forcing for Inaccessible Cardinals}

\author{Shani Cohen and Saharon Shelah}

\date{January 12, 2023}



\thanks{Partially sponsored by the European Research Council grant 338821.}

\thanks{Publication 1085 on the second author's list.}
\begin{abstract}
The two classical parallel concepts of ``small'' sets of the real
line are meagre sets and null sets. Those are equivalent to Cohen
forcing and Random Real forcing for $^{\mathbb{N}}\mathbb{N}$; in
spite of this similarity, the Cohen forcing and Random Real forcing
have very different shapes. One of these differences is in the fact
that the Cohen forcing has an easy natural generalization for $^{\lambda}2$
for regular $\lambda>\aleph_{0}$, corresponding to an extension for
the meagre sets, while the Random Real forcing didn't seem to have
a natural generalization, as Lebesgue measure doesn't have a generalization
for space $2^{\lambda}$ while $\lambda>\aleph_{0}$. The work \cite{Sh:1004}
found a forcing resembling the properties of Random Real forcing for
$2^{\lambda}$ while $\lambda$ is a weakly compact cardinal. Here
we describe, with additional assumptions, such a forcing for $2^{\lambda}$
while $\lambda$ is just an Inaccessible Cardinal; this forcing is
strategically $<\lambda$- complete and satisfies the $\lambda^{+}$-
c.c hence preserves cardinals and cofinalities, however unlike Cohen
forcing, does not add an undominated real.
\end{abstract}

\maketitle
\selectlanguage{british}%
\begin{onehalfspace}
{\large{}\newpage{}}{\large \par}
\end{onehalfspace}

\selectlanguage{english}%

\section*{Introduction}

\medskip
\medskip

There are two classical ways of defining what is a small set of the
real line $^{\omega}2$; the topological definition of a small set
is a \uline{meagre} set, which is a countable union of nowhere
dense sets. The second definition uses measure and defines a set to
be small if it is a \uline{null} set, which means that it has Lebesgue
measure zero.

Both the collection of meagre sets and the collection of null sets
are ideals in the set $^{\omega}2$; the forcing modulo the ideal
of meagre sets is the \uline{Cohen Forcing} while the forcing modulo
the ideal of null sets is \uline{Random Real forcing} \cite{Sh:CR}.

Looking at $\lambda$- reals for $\lambda>\aleph_{0}$, so elements
of the set:

$^{\lambda}2=\{\eta:\eta\mbox{ is a sequence of }0\mbox{'s and }1\mbox{'s of length }\lambda\}$,
there is a natural extension to a Cohen Forcing; that would be a forcing
modulo sets that are $\lambda$- meagre \cite{Jech}. Unlike this
case, Lebesgue measure has no natural extension in $^{\lambda}2$
for regular cardinals $\lambda>\aleph_{0}$, thus there is no generalization
of Random Real forcing for those cardinals.

An important and useful property of Random Real forcing is not adding
a function that is undominated; recall that Cohen Forcing adds $f:\lambda\rightarrow\lambda$
not smaller (meaning, modulo finite set) than any real in the ground
model (where $\lambda$-reals here are functions $\lambda\rightarrow\lambda$,
i.e. members of $^{\lambda}\lambda$). However Random Real forcing
has the property that every ``new'' real (i.e. every element of
$^{\omega}\omega$) is bounded by a real in the ground model. One
of the uses of this property is for cardinal invariants; the bounding
number $\mathfrak{d}$ \cite{Hal} does not change after forcing with
Random Real forcing.

In the paper \cite{Sh:1004}, the second author described a generalization
of the null ideal (meaning, the ideal of Lebesgue measure zero sets)
for a weakly compact cardinal $\lambda$; that was done by constructing
a forcing that has the properties of Random Real forcing in $2^{\lambda}$
for a weakly compact $\lambda$; this result is surprising since there
is no clear similarity in the definition of the forcing in \cite{Sh:1004}
and Random Real forcing.

By ``having the properties of Random Real forcing'' we mean a forcing
for which: (1) the $\lambda^{+}$- chain condition holds and (2) the
forcing is strategically $<\lambda$- complete; by those conditions
it follows that the forcing preserves cardinals and cofinalities when
$\lambda=\lambda^{<\lambda}$. Moreover, any new real added in the
forcing shall be bounded by a real in the ground model, that will
be condition (3): the forcing is $\lambda$- bounding. An additional
important property is symmetry, but it fails by \cite{Sh:1004}.

The purpose of this work is to find a forcing as in \cite{Sh:1004}
for Mahlo, and even any inaccessible cardinal (therefore may be smaller
than the first weakly compact cardinal). In section \ref{Mahlo} we
shall describe a construction for which the properties of Random Real
forcing [\ref{trullymaindef}] hold for any \uline{inaccessible} and in particular
\uline{Mahlo cardinal}; those are cardinals whose existence is
a weaker condition than the existence of a weakly compact cardinal
\cite{Kan}. However compare to \cite{Sh:1004} we need some parameter
$X\subseteq\lambda$ so the definition is not ``pure'' as in \cite{Sh:1004}.

An additional difference from \cite{Sh:1004} is that the large cardinal
property is not enough. We shall assume the existence of a stationary
set that reflects only in inaccessibles and has a diamond sequence.
Note that this demand can be gotten by an easy forcing \cite{CFM}
and if $V=L$ this is equivalent to not being weakly compact. For
a Mahlo cardinal there is a stationary set of inaccessible cardinals
below it so in particular this set reflects only in inaccessibles
and then we still need to assume the existence of diamond sequence
for it. In \cite{Sh:1004}, the main use of the weak compactness was
by reflecting a maximal antichain of conditions to a maximal antichain
in a corresponding forcing for a smaller cardinal; the purpose of
the diamond sequence here will be to overcome this inability.

Furthermore, for convenience we shall assume that the conditions of
the desired forcing are trees that are pruned only in levels of the
stationary set (we demand the stationary set to contain only limit
ordinals). However it is possible to allow pruning in successor levels;
e.g. as long as the pruning is only of a bounded set, when we use
a tree with splitting to $\theta_{\epsilon}=\mathrm{cf}(\theta_{\epsilon}) \in \big[|\epsilon|^{+},\lambda \big)$.

We may like to make our forcing $<\lambda$- complete (rather than
strategically $<\lambda$- complete); this is not clear.

This work is a part of what was promised in \cite{Sh:1004}, the ideas
of the construction where stated in Rutgers in 2011.

We intend to deal later with accessible $\lambda=\lambda^{<\lambda}>\aleph_{0}$,
(under reasonable condition); also we can use $|\epsilon|^{+}$- complete
$D_{\epsilon}$ filter on $\theta_{\epsilon}$ (or $D_{\eta}$ on
$\mathrm{suc}_{\boldsymbol{T}}(\eta)$ when $\lg(\eta)=\epsilon$, as in \cite{Sh:1004},
or Remark \ref{remarkfilter} below).

We also continue \cite{Sh:1004} in \cite{Sh:1100} in work with T. Baumhauer and M. Goldstern \cite{BauGolShe} and in \cite{Sh:E82}.

\noindent
\textbf{Acknowledgements.} We thank Andrzej Ros\l{}anowski for encouragement for finishing the paper
and the referee for doing a wonderful job.

Notation:
\begin{enumerate}
\item we use $\alpha,\beta,\gamma,\delta,\epsilon$ to denote ordinals,
\item we use $\lambda,\mu,\kappa$ to denote cardinals,
\item $^{B}\!A$ denotes the set of functions from $B$ to $A$.
\end{enumerate}
\selectlanguage{british}%
\begin{onehalfspace}
{\large{}\newpage{}}{\large \par}
\end{onehalfspace}

\tableofcontents{}

\begin{onehalfspace}
{\large{}\newpage{}}{\large \par}
\end{onehalfspace}

\selectlanguage{english}%

\section{Preliminaries}

The paper \cite{Sh:1004} showed a method of finding, for a weakly
compact cardinal $\lambda$, a forcing that generalizes the properties
of a Random Real forcing for $\aleph_{0}$. In section \ref{Mahlo}
we add the assumption of a diamond principle and then see a similar
forcing that generalizes the same properties for inaccessible cardinals,
with the assumption that there exists a stationary set that reflects
only in inaccessibles; so in particular for Mahlo cardinals it follows.
Here we show some general definitions that will be used throughout
this paper.
\begin{defn}
\label{def:A-forcing-resembling}A forcing resembling Random Real
forcing for a regular cardinal $\lambda=\lambda^{<\lambda}$ will
be a forcing for which the following conditions hold:
\begin{enumerate}
\item The forcing is not trivial and the $\lambda^{+}$- chain condition
holds.
\item The forcing is $<\lambda$- strategically complete.
\item \label{forcingisbounding}The forcing is $\lambda$- bounding.
\item The forcing does not add $\lambda$- Cohen reals (follows from \ref{forcingisbounding}).
\end{enumerate}
\end{defn}
This definition reflects the properties of Random Real forcing in
the case of $\lambda=\aleph_{0}$.
\begin{rem}
At this point we shall ignore another desired property, symmetry.
This property states that for all $\eta_{1},\eta_{2}$ and a model
$M$: $\eta_{1}$ is generic over $M$ and $\eta_{2}$ is generic
over $M[\eta_{1}]$ if and only if $\eta_{2}$ is generic over $M$
and $\eta_{1}$ is generic over $M[\eta_{2}]$. By \cite{Sh:1004}
it fails.
\end{rem}
Now we can define the terms used for definition \ref{def:A-forcing-resembling}:
\begin{defn}
Let $\alpha$ be an ordinal.
\begin{enumerate}
\item For a forcing notion $\mathbb{P}$ and condition $p\in\mathbb{P}$,
we define a game $\Game_{\alpha}(p,\mathbb{P})$ as follows. A play
of the game has $\alpha$ moves and for $\beta<\alpha$ first the
player COM chooses a condition $p_{\beta}\in\mathbb{P}$ such that:

\begin{enumerate}
\item $p\leq p_{\beta}$.
\item For all $\gamma<\beta$ it holds that $q_{\gamma}\leq p_{\beta}$.
\end{enumerate}

Next, the player INC plays and chooses $q_{\beta}\in\mathbb{P}$ such
that $p_{\beta}\leq q_{\beta}$.

The player \uline{COM wins the game} if she survived; i.e. had
a legal move for all $\beta<\alpha$.

\item A forcing $\mathbb{P}$ is said to be \uline{strategically complete
in $\alpha$} (or $\alpha$- strategically complete) if for all $p\in\mathbb{P}$
it holds that in the game $\Game_{\alpha}(p,\mathbb{P})$ between
players COM and INC, player COM has a winning strategy.
\end{enumerate}
\end{defn}

\begin{defn}
A forcing $\mathbb{P}$ is \uline{$<\lambda$- strategically complete}
if it is $\alpha$- strategically complete for all $\alpha<\lambda$.
\end{defn}

\begin{defn}
For a cardinal $\lambda$, a forcing $\mathbb{P}$ will be called
\uline{$\lambda$-bounding} when the following holds: $$\Vdash_{\mathbb{P}} \big(\forall f:\lambda\rightarrow\lambda \big) \big(\exists g\in({}^{\lambda}\lambda)^{V} \big) \big(\forall\alpha<\lambda \big) \big[ f(\alpha)\leq g(\alpha) \big]$$
\end{defn}

\begin{defn}
A set of ordinals $S$ will be called \uline{tenuous} (or ``nowhere
stationary'' as in \cite{Sh:1004}) if for each ordinal $\delta$
of uncountable cofinality, the set $S\restriction\delta$ is not a
stationary set in $\delta$.
\end{defn}

\begin{defn}
Let $\lambda$ be a cardinal and $S\subseteq\lambda$ a stationary
set of $\lambda$. Then $S$ is said to be \uline{non-reflecting}
when for each ordinal $\delta<\lambda$ of cofinality $>\aleph_{0}$
the set $S\restriction\delta$ is not stationary in $\delta$.\end{defn}
\begin{rem}
\label{stationarynotreflecting}Let $\lambda$ be a cardinal and
let $S_{\ast}$ be a non-reflecting stationary subset of $\lambda$.
Then the set $S\subseteq S_{\ast}$ is tenuous if and only if $S$
is not stationary.
\begin{claim}
\label{stationaryunions}Let $\lambda$ be a cardinal and $S_{\ast}$
be a non-reflecting stationary subset of $\lambda$.
\begin{enumerate}
\item \label{unionnonstat}If $\bar{S}=\langle S_{i}:i<i(\ast)\rangle$
is such that for all $i<i(\ast)$, $S_{i}\subseteq\lambda$ is a non-stationary
with $i(\ast)<\mathrm{cf}(\lambda)$; \uline{then} $S=\bigcup\limits_{i<i(\ast)}S_{i}$
is not stationary.
\item \label{unionnonten}If $\bar{S}=\langle S_{i}:i<i(\ast)\rangle$
is such that for all $i<i(\ast)$, $S_{i}\subseteq S_{\ast}$ is a
tenuous set with $i(\ast)<\mathrm{cf}(\lambda)$; \uline{then} $S=\bigcup\limits_{i<i(\ast)}S_{i}$
is tenuous.
\end{enumerate}
\end{claim}
\end{rem}
\begin{proof}
We see:
\begin{enumerate}
\item For each $i<i(\ast)$, there is a club $E_{i}$ such that $S_{i}\cap E_{i}=\varnothing$
(as $S_{i}$ is not stationary); so let $E=\bigcap\limits_{i<i(*)}E_{i}$.
$E$ is a club in $\lambda$, as the intersection of $i(\ast)<\mathrm{cf}(\lambda)$
clubs. In addition, $S\cap E=\varnothing$, thus $S$ is not stationary.
\item From clause \ref{unionnonstat}, $S$ is not stationary. In addition
for each $\alpha<\lambda$, $S_{\ast}\cap\alpha$ is non-stationary
hence so does $S\cap\alpha$, as a subset of it.
\end{enumerate}
\end{proof}
\selectlanguage{british}%
\begin{onehalfspace}
{\large{}\newpage{}}{\large \par}
\end{onehalfspace}

\selectlanguage{english}%

\section{\label{Mahlo}New $\lambda$-Real for Inaccessible Cardinal $\lambda$}

To find a forcing resembling Random Real forcing for Mahlo Cardinal,
we need to add an additional assumption to those of the weakly compact
cardinal case in \cite{Sh:1004}; the new assumption will be a diamond
sequence indexed on a stationary set of inaccessible cardinals (a
stationary set of inaccessibles exists for a Mahlo Cardinal). For
the more general case of any Inaccessible Cardinal, there is still
a need to assume the existence of a diamond sequence; however here
it will be indexed on a stationary set that only reflects in inaccessible
cardinals. Those two cases are unified here, dealing with an Inaccessible
Cardinal with a stationary set that only reflects in inaccessible
cardinals; a Mahlo Cardinal will be a special case of this.

\subsection{Useful Definitions}
\begin{defn}
\label{finestructure}A \uline{good structure} $\mathfrak{r}$
contains:
\begin{enumerate}
\item An inaccessible cardinal $\lambda=\lambda_{\mathfrak{r}}$.
\item A stationary set $S_{\ast}=S_{\ast}^{\mathfrak{r}}\subseteq\lambda$
of strong limit cardinals, such that if $S_{\ast}\cap\delta$ is stationary
in $\delta$ then $\delta$ is inaccessible.
\item An increasing sequence of cardinals $\bar{\theta}=\bar{\theta}_{\mathfrak{r}}=\langle\theta_{\epsilon}:\epsilon<\lambda\rangle$
such that for all $\epsilon<\lambda$: $2\leq\theta_{\epsilon}<\lambda$
and if $\epsilon\in S_{\ast}$ then for all $\zeta<\epsilon$, $\theta_{\zeta}<\epsilon$.
\item \label{diamondsequence}We assume the diamond principle for $S_{\ast}$,
$\Diamond_{S_{\ast}}$, and let $\bar{X}=\bar{X}_{\mathfrak{r}}$
be a sequence witnessing it, i.e. $\bar{X}=\langle X_{\delta}:\delta\in S_{\ast}\rangle$;
$X_{\delta}\subseteq\mathscr{H}(\lambda)$.
\end{enumerate}
\end{defn}
\begin{rem}
Observe:
\begin{enumerate}
\item For $\lambda$ Mahlo there is a stationary set $S_{\ast}\subseteq\lambda$
that only contains inaccessible cardinals, thus in particular its
reflection will only be in inaccessible cardinals.
\item For $\lambda$ inaccessible that isn't Mahlo, a non-reflecting stationary
set can be added by a forcing that uses initial segments as in \cite{CFM}.
\item It is possible to assume that $S_{\ast}$ is a set of just limit ordinals
(maybe not strong limit) and the only difference will be that for
all $\delta\in S_{\ast}$ the forcing $\mathbb{Q}_{\delta}$ (will
be defined later) will have the $|\boldsymbol{T}_{\!<\delta}|^{+}$-
chain condition rather than the $\delta^{+}$- chain condition as
we have here; however the forcing $\mathbb{Q}_{\lambda}$ will still
have the $\lambda^{+}$- chain condition. 
\item Concerning \ref{finestructure}(\ref{diamondsequence}), the standard
phrasing of $\Diamond_{S}$ is ``there is $\langle A_{\alpha}:\alpha\in S\rangle$,
$A_{\alpha}\subseteq\alpha$ such that for every $A\subseteq\alpha$
the set $\{\alpha\in S:A\cap\alpha=A_{\alpha}\}$ is a stationary
subset of $\lambda$''. However given a sequence as above, and $h$
an one-to-one function from $\lambda$ onto $\mathscr{H}(\lambda)$
(they are of the same cardinality because $\lambda=\lambda^{<\lambda}$
follows from ``$\lambda$ is inaccessible''). Let $E=\{\mu<\lambda:h\mbox{ maps }\mathscr{H}(\mu)\mbox{ onto }\mu\}$.
It is a club of $\lambda$ because $\mu<\lambda\Rightarrow2^{\mu}<\lambda$.
Lastly let $X_{\delta}\subseteq\delta$ be $\{\alpha<\delta:h(\alpha)\in A_{\delta}\}$,
easily it is as required. Inversely, starting from the $X_\delta$-s we can choose suitable $A_\delta$-s.
\end{enumerate}
\end{rem}

\begin{rem}
When $S_{\ast}$ is non-reflecting, the proofs are simpler.
\end{rem}

Next, the forcing will be defined in several steps; those will be
tree forcings for each $\delta\in S_{\ast}\cup\{\lambda\}$. First,
we shall define the ``biggest'' forcing $\mathbb{Q}_{\delta}^{0}$;
later we will define two additional forcing $\mathbb{Q}_{\delta}\subseteq\mathbb{Q}_{\delta}'\subseteq\mathbb{Q}_{\delta}^{0}$.
For each of those forcing the forcing relation will be of inverse
inclusion.

\begin{defn}
Given a good structure $\mathfrak{r}$, we shall define for each $\alpha\leq\lambda$
the collection of vertices of level $\alpha$: $\boldsymbol{T}_{\alpha}=\prod\limits_{\epsilon<\alpha} \theta_{\epsilon}$;
for $\alpha\leq\lambda$ we will define the complete tree up to $\alpha$
to be the union of those sets: $\boldsymbol{T}_{<\alpha}=\bigcup\{\boldsymbol{T}_{\beta}:\beta<\alpha\}$.\end{defn}
\begin{rem}
We assume we have a good structure $\mathfrak{r}$ until the end of
section \ref{Mahlo}.
\end{rem}

\begin{conv}We let:
\begin{enumerate}
    \item For all $\delta_{1}<\delta\leq\lambda$ and $\nu\in\boldsymbol{T}_{\delta}$ let $\nu\restriction\delta_{1}$ the restriction of $\nu$ to $\delta_{1}$.
    
    \item For each $\delta\in S_{\ast}\cup\{\lambda\}$ and a set $u\subseteq\boldsymbol{T}_{\!<\delta}$ we write $$\lim_{\delta}(u)=\{\nu\in\boldsymbol{T}_{\delta}:\forall\alpha<\delta,\ \nu\restriction\alpha\in u\}.$$
    
    \item For all $\delta\leq\lambda$ and a set $u\subseteq\boldsymbol{T}_{\!<\delta}$, for $\delta_{1}<\delta$ we shall write $u\restriction\delta_{1}=u\cap\boldsymbol{T}_{<\delta_{1}}$.
    
    \item Assume $\alpha<\delta$ and $u\subseteq\boldsymbol{T}_{<\alpha}$ is a tree: non-empty set closed under taking initial segments. Let $\eta\in u$ be some node; we write $$u^{[\eta]}=\{\nu\in u:\eta\trianglelefteq\nu\vee\nu\triangleleft\eta\}.$$
\end{enumerate}
\end{conv}
\begin{defn}
\label{veryfirstforcing}We can now define the forcing $\mathbb{Q}_{\delta}^{0}$
for each $\delta\in S_{\ast}\cup\{\lambda\}$.
\begin{enumerate}
\item \label{veryfirstforcing1}A condition in the forcing will be a
tree $p\subseteq\boldsymbol{T}_{\!<\delta}$, such that:

\begin{enumerate}
\item \label{veryfirstforcinga}There is a trunk $\mathrm{tr}(p)$; this is the
unique element $\eta\in p$ with the following properties:

\begin{enumerate}
\item \label{beingatrunk(i)}For all $\nu\in p$ it holds that $\nu\trianglelefteq\eta$
or $\eta\trianglelefteq\nu$.
\item For every $\eta'$ with the property \ref{beingatrunk(i)}, we
have that $\eta'\trianglelefteq\eta$.
\end{enumerate}
\item \label{veryfirstforcingb}For each $\eta\in p$ there is a $\nu\in\lim_{\delta}(p)$
with $\eta\triangleleft\nu$.
\item \label{veryfirstforcingc}For $\mathrm{tr}(p)\trianglelefteq\eta\in p$;
$\{j\in\theta_{\lg(\eta)}:\eta\overset{\frown}{}\langle j\rangle\in p\}=\theta_{\lg(\eta)}$.
\item \label{veryfirstforcingd}The set 
$$S_{p} = \big\{\delta_{1} \in \big( \lg(\mathrm{tr}(p)),\delta \big) : \lim_{\delta_{1}}(p\restriction\delta_{1}) \not\subseteq p \big\}$$
is tenuous subset of $S_{\ast}$; we call this set the witness set.
\end{enumerate}
\item For all $p,q\in\mathbb{Q}_{\delta}^{0}$ we say that $p\leq q$ if
and only if $p\supseteq q$.
\end{enumerate}
\end{defn}

\begin{rem}
We can think of a tree $p\in\mathbb{Q}_{\delta}^{0}$ for $\delta\in S_{\ast}\cup\{\lambda\}$
as a complete tree from the level $\lg(\mathrm{tr}(p))$ up that we are pruning:
in successor levels we are not allowed to prune. On limit levels we
are allowed to prune the tree only if the level is an ordinal in $S_{p}$,
so in most limit levels we take all the limits while in stages in
$S_{p}$ we are allowed to cut as much as we want as long as \ref{veryfirstforcingb}
holds; so there will be a continuation to each node in each level
higher than its length.
\end{rem}

\begin{rem}
\label{remarkfilter}An alternative definition can be such that in
successor levels there might be prunings, as long as those are not
too big, that is, $\{j\in\theta_{\lg(\eta)}:\eta\overset{\frown}{}\langle j\rangle\notin p\}$
is bounded in $\theta_{\lg(\eta)}$ (when $\mathrm{cf}(\theta_{\lg(\eta)})>\lg(\eta)$)
or even belong to $D_{\eta}$, a $|\lg(\eta)|^{+}$- complete filter
on $\theta_{\lg(\eta)}$. Not a serious difference.
\end{rem}

\begin{claim}
\label{generalstuff}For all $\delta\in S_{\ast}\cup\{\lambda\}$,
the forcing $\mathbb{Q}_{\delta}^{0}$ has the following properties:
\begin{enumerate}
\item \label{generalstuff1}The whole tree $\boldsymbol{T}_{\!<\delta}\in\mathbb{Q}_{\delta}^{0}$
and is weaker than any other condition in the forcing $\mathbb{Q}_{\delta}^{0}$.
\item \label{generalstuff2}If $p\in\mathbb{Q}_{\delta}^{0}$ and $\eta\in p$,
then $p^{[\eta]}\in\mathbb{Q}_{\delta}^{0}$ and $p\leq_{\mathbb{Q}_{\delta}^{0}}p^{[\eta]}$.
\item \label{generalstuff3}Let $\epsilon<\delta$; then the set $\{(\boldsymbol{T}_{\!<\delta})^{[\eta]}:\eta\in\boldsymbol{T}_{\epsilon}\}$
is a maximal antichain of the forcing $\mathbb{Q}_{\delta}^{0}$.
\item \label{generalstuff4}Let $p\in\mathbb{Q}_{\delta}^{0}$ and $\epsilon<\delta$,
then $\{p^{[\eta]}:\eta\in p\cap\boldsymbol{T}_{\epsilon}\}$ is a
maximal antichain above $p$ (if $\epsilon\leq\lg(\mathrm{tr}(p))$, $p=p^{[\eta]}$
and this set is a singleton).\end{enumerate}
\begin{proof}
Let $\delta\in S_{\ast}\cup\{\lambda\}$:
\begin{enumerate}
\item Trivial.
\item Trivial.
\item Let $\epsilon<\delta$; the set is an antichain since for any $\eta\neq\nu\in\boldsymbol{T}_{\epsilon}$,
clearly $(\boldsymbol{T}_{\!<\delta})^{[\eta]}$ and $(\boldsymbol{T}_{\!<\delta})^{[\nu]}$
are not compatible. Let $p\in\mathbb{Q}_{\delta}^{0}$ and let $\eta\in p\cap\boldsymbol{T}_{\epsilon}$
be a node. There exists such a node recalling clauses \ref{veryfirstforcinga}
and \ref{veryfirstforcingb} of definition \ref{veryfirstforcing},
then $p^{[\eta]}\in\mathbb{Q}_{\delta}^{0}$ by previous clause; $p\leq_{\mathbb{Q}_{\delta}^{0}}p^{[\eta]}$
and clearly $(\boldsymbol{T}_{\!<\delta})^{[\eta]}\leq_{\mathbb{Q}_{\delta}^{0}}p^{[\eta]}$;
thus $\{(\boldsymbol{T}_{\!<\delta})^{[\eta]}:\eta\in\boldsymbol{T}_{\epsilon}\}$
is a maximal antichain in $\mathbb{Q}_{\delta}^{0}$. 
\item Similar.
\end{enumerate}
\end{proof}
\end{claim}
Next, define a structure that will fulfill the roll of using the diamond
principle; the structure will be a collection of objects that contain
elements that are antichains with additional properties. In the weakly
compact case \cite{Sh:1004} there was an important roll for the maximal
antichains; in the proof of the $\lambda$-bounding of the forcing
there was a maximal antichain that reflected to an antichain in the
forcing corresponding to a smaller cardinal.

In the inaccessible case which we are dealing with here, we will have
to use diamond to gain a similar property. Each element is an antichain
in the forcing $\mathbb{Q}_{\delta}^{0}$.
\begin{defn}
\label{Chi}For any ordinal $\delta\in S_{\ast}\cup\{\lambda\}$,
$\Xi_{\delta}$ will be the collection of objects $\bar{q}$, for
which the following conditions hold:
\begin{enumerate}
    \item $\bar{q}=\langle q_{\eta}:\eta\in\Lambda\rangle$,
    
    \item $\Lambda\subseteq\boldsymbol{T}_{\!<\delta}$,
    
    \item for each $\eta\in\Lambda$ it holds that $q_{\eta}\in\mathbb{Q}_{\delta}^{0}$ and $\eta=\mathrm{tr}(q_{\eta})$,
    
    \item if $\eta,\nu\in\Lambda$ and $\eta\neq\nu$ then $\eta=\mathrm{tr}(q_{\eta})\notin q_{\nu}\vee\nu=\mathrm{tr}(q_{\nu})\notin q_{\eta}$,
    
    \item the union of all the conditions from the set will be an element in the forcing: $r_{\bar{q}}^{\ast}=\{\rho\in\boldsymbol{T}_{\!<\delta}:(\exists\eta\in\Lambda)(\rho\in q_{\eta}) \} \in\mathbb{Q}_{\delta}^{0}$.
\end{enumerate}

\begin{defn}
For all $\delta\in S_{\ast}\cup\{\lambda\}$ and $\bar{q}\in\Xi_{\delta}$,
we define the \uline{coder} $X_{\bar{q}}$: $X_{\bar{q}}=\{(\eta,\nu):(\eta\in\Lambda)\wedge(\nu\in q_{\eta})\}\subseteq\mathscr{H}(\delta)$.
\begin{defn}
\label{weaklysuccessfull}Let $\delta\in S_{\ast}$; we call $\delta$
\uline{weakly successful} when there exists $\bar{q}\in\Xi_{\delta}$
with $X_{\bar{q}}=X_{\delta}$, recalling $X_{\delta}$ is from the
good structure $\mathfrak{r}$ defined in clause \ref{diamondsequence}
of \ref{finestructure}.
\end{defn}
\end{defn}
\end{defn}

\begin{claim}
For a weakly successful $\delta\in S_{\ast}$, the $\bar{q}$ of definition
\ref{weaklysuccessfull} is unique.

\begin{proof}
Observe that the coder $X_{\bar{q}}$ has all the information on $\bar{q}$,
therefore such a $\bar{q}$ must be unique.
\end{proof}
\end{claim}
\begin{defn}\label{whattodowithweaklysuccessful}
\begin{enumerate}
    \item For a weakly successful $\delta\in S_{\ast}$:

    \begin{itemize}
        \item look at the unique sequence $\bar{q} = \langle q_{\eta} : \eta \in \Lambda\rangle$ for which $X_{\bar{q}}=X_{\delta}$ and write $\Lambda_{\delta}^{\ast}=\Lambda$; for all $\eta\in\Lambda_{\delta}^{\ast}$ let $q_{\delta,\eta}^{\ast} = q_{\eta}$ and lastly $\bar{q}_{\delta}^{\ast} = \langle q_{\delta,\eta}^{\ast} : \eta\in\Lambda_{\delta}^{\ast}\rangle$,

        \item let $r_{\delta}^{\ast}=r_{\bar{q}_{\delta}^{\ast}}^{\ast}=\{\rho\in\boldsymbol{T}_{\!<\delta}:(\exists\nu\in\Lambda_{\delta}^{\ast})(\rho\in q_{\delta,\nu}^{\ast})\}$,

        \item finally, for each $\eta\in\Lambda_{\delta}^{\ast}$ and $\nu\in\boldsymbol{T}_{\!<\delta}$ with: $\eta\trianglelefteq\nu\in q_{\delta,\eta}^{\ast}$, let $q_{\delta,\nu}^{\ast}=(q_{\delta,\eta}^{\ast})^{[\nu]}$.
    \end{itemize}

    \item \label{23(2)}For $\delta\in S_{\ast}$ which is not weakly successful, let $r_{\delta}^{\ast}=\boldsymbol{T}_{\!<\delta}$ and for all $\eta\in r_{\delta}^{\ast}$: $q_{\delta,\eta}^{\ast}=(r_{\delta}^{\ast})^{[\eta]}$.
\end{enumerate}
\end{defn}

Now we can use the $\bar{X}$, being a diamond sequence:
\begin{claim}
\label{useanti}For every $\bar{q} = \langle q_{\eta} : \eta \in \Lambda\rangle \in \Xi_{\lambda}$
there is a stationary set of weakly successful $\delta\in S_{\ast}$
for which $\bar{q}_{\delta}^{\ast} = \langle q_{\eta} \cap \boldsymbol{T}_{\!<\delta} : \eta\in\Lambda\cap\boldsymbol{T}_{\!<\delta}\rangle$.
\begin{proof}
Recall that $\bar{X}$ is a diamond sequence, therefore for the set
$X_{\bar{q}}$, there is a stationary set of $\delta\in S_{\ast}$
for which $X_{\delta}=X_{\bar{q}}\cap\mathscr{H}(\delta)=X_{\bar{q}}\cap(\boldsymbol{T}_{\!<\delta}\times\boldsymbol{T}_{\!<\delta})$,
from the definition of the coder, and as $X_{\delta}$ is the coder
of $\bar{q}_{\delta}^{\ast}$ the conclusion follows: $\bar{q}_{\delta}^{\ast}=\langle q_{\eta}\cap\boldsymbol{T}_{\!<\delta}:\eta\in\Lambda\cap\boldsymbol{T}_{\!<\delta}\rangle$,
where $\bar{q}_{\delta}^{\ast}\in\Xi_{\delta}$ witness that $\delta$
is weakly successful.
\end{proof}
\end{claim}

\subsection{Defining the Main Forcing}
\begin{rem}
\label{beforeforcing}Below the main forcing will be defined, however
prior to the definition we would like to state the properties that
this forcing is expected to have; this remark is meant to describe
the general structure of the forcings $\mathbb{Q}_{\delta}'$ amd
$\mathbb{Q}_{\delta}$ for each $\delta\in S_{\ast}\cup\{\lambda\}$.
\begin{enumerate}
\item We would like those forcing to be subforcings of $\mathbb{Q}_{\delta}^{0}$
(but not nesseccarily complete subforcings), where $\mathbb{Q}_{\delta}\subseteq\mathbb{Q}_{\delta}'\subseteq\mathbb{Q}_{\delta}^{0}$.
\item For a condition $p\in\mathbb{Q}_{\delta}$ and a node $\eta\in p$,
we have $p^{[\eta]}\in\mathbb{Q}_{\delta}$; the same holds for $\mathbb{Q}_{\delta}'$.
\item The complete tree $\boldsymbol{T}_{\!<\delta}$ belongs to $\mathbb{Q}_{\delta}$;
so in particular it belongs to $\mathbb{Q}_{\delta}'$.
\end{enumerate}
\end{rem}

\begin{rem}
We are now ready to finally define the desired forcings. The following
is the main definition of the forcing. Pedantically, the induction
in definition \ref{trullymaindef} should be carried together with
the proof of Claim \ref{someproperties}.
\end{rem}

\begin{defn}
\label{trullymaindef}The definition is inductive on $\delta$;
we will define the subforcings of $\mathbb{Q}_{\delta}^{0}$: $\mathbb{Q}_{\delta}$
and $\mathbb{Q}_{\delta}'$ for all $\delta\in S_{\ast}\cup\{\lambda\}$,
in addition we will define the term successful for ordinals $\delta\in S_{\ast}$
and for each $\eta\in\boldsymbol{T}_{\!<\delta}$ and a tenuous $S\subseteq S_{\ast}\cap\delta$
we will define $p_{\eta,\delta,S}^{\ast}\in\mathbb{Q}_{\delta}^{0}$,
so we have to verify this, see claim \ref{someproperties} below.
\begin{enumerate}
    \item $\forall p\in\mathbb{Q}_{\delta}^{0},\quad p\in\mathbb{Q}_{\delta}'$ iff $\forall\delta_{1}\in\delta\cap S_{\ast}$, $\lg(\mathrm{tr}(p))<\delta_{1}\Rightarrow p\restriction\delta_{1}\in\mathbb{Q}_{\delta_{1}}$,
    
    So the forcing $\mathbb{Q}_{\delta}'$ is derived from the forcings $\mathbb{Q}_{\delta_{1}}$ for $\delta_{1}\in\delta\cap S_{\ast}$.

    \item $\forall p\in\mathbb{Q}_{\delta}^{0},\quad p\in\mathbb{Q}_{\delta}$ iff $p=p_{\eta,\delta,S}^{\ast}$ for some $\eta,S$ satisfying $\eta\in\boldsymbol{T}_{\!<\delta}$, $S\subseteq S_{\ast}\cap\delta$ is tenuous; such $p_{\eta,\delta,S}^{\ast}$ is defined below in clause \ref{maindef},
    
    \item We call $\delta<\lambda$ \uline{successful} when it is weakly successful and in addition: $r_{\delta}^{\ast},q_{\delta,\eta}^{\ast}\in\mathbb{Q}_{\delta}'$ for all $\eta\in\Lambda_{\delta}^{\ast}$.

    \uline{Explanation:} The successful ordinals represent the levels in which there will be a special pruning, determined by the diamond condition, so there is a ``control'' on the conditions defined uniquely, and in relation to $r_{\delta_{1}}^{\ast}$ of the corresponding levels $\delta_{1}$.

    \item \label{maindef}We assume that the forcings $\mathbb{Q}_{\delta'}',\mathbb{Q}_{\delta'}$ are defined for all $\delta'\in S_{\ast}\cap\delta$; the condition $p_{\eta',\delta',S'}^{\ast}$ is defined for all $\eta'\in\boldsymbol{T}_{<\delta'}$ and tenuous $S'\subseteq S_{\ast}\cap\delta'$; we shall define $p_{\eta,\delta,S}^{\ast}\in\mathbb{Q}_{\delta}^{0}$ for $\eta,S$ as above in the following way:
    \begin{enumerate}
        \item \label{maindefa}If $\sup(S)\leq\lg(\eta)$ then $p_{\eta,\delta,S}^{\ast}=\boldsymbol{T}_{\!<\delta}^{[\eta]}$.
        
        \item \label{maindefb}If $\sup(S)>\lg(\eta)$ and $S$ has no maximal element, then for each $\nu\in\boldsymbol{T}_{\!<\delta}$ it holds that $\nu\in p_{\eta,\delta,S}^{\ast}$ if and only if one of the following conditions holds:

        \begin{enumerate}
            \item \label{maindefb(i)}$\nu\trianglelefteq\eta$,
            
            \item \label{maindefb(ii)}$\eta\lhd\nu$ and there exists $\lg(\nu)<\delta_{1}\in S$ such that $\nu\in p_{\eta,\delta_{1},S\cap\delta_{1}}^{\ast}$,
            
            \item \label{maindefb(iii)}$\eta\lhd\nu$, $\lg(\nu)\geq\sup(S)$ and for all $\delta_{1}\in S\setminus(\lg(\eta)+1)$ and $\zeta<\delta_{1}$ it holds that $\nu\restriction\zeta\in p_{\eta,\delta_{1},S\cap\delta_{1}}^{\ast}$.
        \end{enumerate}
        
        \item \label{maindefc}If $\sup(S)>\lg(\eta)$ and $S$ has a last element $\delta_{1}<\delta$, such that $\delta_{1}$ is \uline{not successful}, then for each $\nu\in\boldsymbol{T}_{\!<\delta}$ it holds that $\nu\in p_{\eta,\delta,S}^{\ast}$ if and only if one of the following holds:
        \begin{enumerate}
            \item \label{maindefc(i)}$\lg(\nu)<\delta_{1}\wedge\nu\in p_{\eta,\delta_{1},S\cap\delta_{1}}^{\ast}$,
            
            \item\label{maindefc(ii)} $\lg(\nu)\geq\delta_{1}\wedge\nu\restriction\delta_{1}\in\lim_{\delta_{1}}(p_{\eta,\delta_{1},S\cap\delta_{1}}^{\ast})$ recalling \ref{whattodowithweaklysuccessful}(\ref{23(2)}).
        \end{enumerate}
        
        \item \label{maindefd}If $\sup(S)>\lg(\eta)$ and $S$ has a last element $\delta_{1}<\delta$ such that $\delta_{1}$ is \uline{successful}, then for each $\nu\in\boldsymbol{T}_{\!<\delta}$ it holds that $\nu\in p_{\eta,\delta,S}^{\ast}$ if and only if one of the following hold:

        \begin{enumerate}
            \item \label{maindefd(i)}$\lg(\nu)<\delta_{1}\wedge\nu\in p_{\eta,\delta_{1},S\cap\delta_{1}}^{\ast}$,
            
            \item \label{maindefd(ii)} $\lg(\nu)>\delta_{1}\wedge\nu\restriction\delta_{1}\in p_{\eta,\delta,S}^{\ast}\cap\boldsymbol{T}_{\delta_{1}}$ according to the definition of $p_{\eta,\delta,S}^{\ast}\cap\boldsymbol{T}_{\delta_{1}}$ in the following clause,
            
            \item \label{maindefd(iii)}On the level of the last element of $S$, the process is more interesting; for $\lg(\nu)=\delta_{1}$:

            \begin{enumerate}
                \item \label{maindefd(iii)(A)}If $\nu\notin\lim_{\delta_{1}}(r_{\delta_{1}}^{\ast})$ then $\nu\in\lim_{\delta_{1}}(p_{\eta,\delta_{1},S\cap\delta_{1}}^{\ast})$,
                
                \item \label{maindefd(iii)(B)}If $\nu\in\lim_{\delta_{1}}(r_{\delta_{1}}^{\ast})$ then 
                $$\nu\in \Big(\bigcup\{\lim_{\delta_{1}}(q_{\delta_{1},\eta'}^{\ast}):\eta'\in\Lambda_{\delta_{1}}^{\ast}\} \Big) \cap \lim_{\delta_{1}}(p_{\eta,\delta_{1},S\cap\delta_{1}}^{\ast}).$$
            \end{enumerate}
        \end{enumerate}
    \end{enumerate}
\end{enumerate}
\end{defn}

\uline{Explanation:} 
\begin{enumerate}
    \item [(0)]\setcounter{enumi}{0}Why do we arrive to this definition? In \cite{Sh:1004}, we start with $p=(\boldsymbol{T}_{\lambda})^{[\eta]}$, as in the $\lambda$-Cohen forcing, but add the following pruning: for each condition for some tenuous subset $S\subseteq S_{\ast}$ consisting of inaccessible cardinals, for each $\delta \in S$ we have a set $\Lambda_{\delta}$ of $\leq\delta$ maximal anti-chains of $\mathbb{Q}_{\delta}$.
    The pruning is: for each $\delta\in S$ we omit the $\eta$ of level $\delta$ avoiding at least one of those maximal antichains (and all proper initial segments of which are in the condition). How is this useful in proving the $^{\lambda}\lambda$ case? In \cite{Sh:1004}
it is done by having reflection of the property of being a maximal
antichain; this naturally require reflecting ``a
set of conditions is a maximal antichain in the forcing''. Naturally
each maximal antichain in $\Lambda_{\delta}$ comes from starting
with a maximal antichain in the whole forcing representing a name
of an ordinal $<\lambda$, and demanding that a condition from the
maximal antichain with trunk of length $<\delta$ will be in the generic
sets; i.e. our condition forces this. How can we do this without the
weak compactness assumption? By the diamond on $S_{\ast}$, we guess
a ``poor man maximal antichain,''
those are the $\bar{q}_{\delta}^{\ast}$-s. They look like an approximation
to a maximal antichain inside $r_{\delta}^{\ast}$, but usually are
far from being a maximal antichain. So we intend to make them maximal
antichain above $r_{\delta}^{\ast}$ by ``decree;''
or you may say by definition. This of course change the proof in several
ways.

    \item For the level $\delta\in S_{\ast}\cup\{\lambda\}$, the idea is that
the main forcing $\mathbb{Q}_{\delta}$, is a full tree that is only
being pruned in levels that are in the matching tenuous set, the idea
is to have enough ``thickness'' to achieve the required completeness
and more. In addition, the conditions are unique relative to the tenuous
set and the trunk and by their definition closely related to the diamond
sequence, intuitively, that is needed for the forcing to be bounding.

    \item Note that we gave a definition of $p_{\eta,\delta,S}^{\ast}$ in \ref{trullymaindef}(\ref{maindef}),
it will be defined as a subset of $\boldsymbol{T}_{\!<\delta}$, but
it is really a member of $\mathbb{Q}_{\delta}^{0}$ as is proved in
Claim \ref{someproperties}.
\end{enumerate}

\begin{defn}
For $\delta\in S_{\ast}\cup\{\lambda\}$, define $\name{\eta}_{\delta}$
to be a $\mathbb{Q}_{\delta}$-name: 
$$\name{\eta}=\bigcup\{\mathrm{tr}(p):p\in\Name{G}_{\mathbb{Q}_{\delta}}\}$$
where $\Name{G}_{\mathbb{Q}_{\delta}}$ is a canonical $\mathbb{Q}_{\delta}$-name for the generic filter. If $\delta$ is clear from the context,
we may write $\name{\eta}$ instead $\name{\eta}_{\delta}$.
\end{defn}

\begin{claim}
\label{someproperties}For all $\delta\in S_{\ast}\cup\{\lambda\}$,
$\eta\in\boldsymbol{T}_{\!<\delta}$ and tenuous $S\subseteq S_{\ast}\cap\delta$;
if $p=p_{\eta,\delta,S}^{\ast}$ \uline{then}:
\begin{enumerate}
    \item \label{properties1}If $\delta_{0}\in S_{\ast}\cap\delta$ is such
that $\eta\in\boldsymbol{T}_{<\delta_{0}}$ then $p_{\eta,\delta,S}^{\ast}\restriction\delta_{0}=p_{\eta,\delta_{0},S\cap\delta_{0}}^{\ast}$.

    \item \label{properties2}$\eta$ is the trunk of $p_{\eta,\delta,S}^{\ast}$.
    
    \item \label{properties3} $p_{\eta,\delta,S}^{\ast}\in\mathbb{Q}_{\delta}^{0}$; Moreover, $p_{\eta,\delta,S}^{\ast}\in\mathbb{Q}_{\delta}'$.
    
    \item \label{properties4}The tenuous set $S$ contains the set of pruning levels corresponding to the condition $p=p_{\eta,\delta,S}^{\ast}$; that is, $S_{p}\subseteq S$.

    \item \label{properties5}In addition, $\mathbb{Q}_{\delta}\subseteq\mathbb{Q}_{\delta}' \subseteq \mathbb{Q}_{\delta}^{0}$ and $\mathbb{Q}_{\lambda}=\mathbb{Q}_{\lambda}'$.
\end{enumerate}

\begin{proof}
We prove all statements of the claim by simultaneous induction on
the ordinals $\delta\in S_{\ast}\cup\{\lambda\}$. Assume that the
claim is true for $\mathbb{Q}_{\delta'}$ (so for all conditions $p_{\eta',\delta',S'}^{\ast}$
where $\delta'\in\delta\cap S_{\ast}$, $\eta'\in\boldsymbol{T}_{<\delta'}$
and $S'\subseteq S_{\ast}\cap\delta'$). We will now prove it for $p=p_{\eta,\delta,S}^{\ast}$
where $\eta\in\boldsymbol{T}_{\!<\delta}$ and $S\subseteq S_{\ast}\cap\delta$:
\begin{enumerate}
    \item Assume that $\delta_{0}\in S_{\ast}\cap\delta$ is such that $\eta\in\boldsymbol{T}_{<\delta_{0}}$;

    \begin{enumerate}
        \item \label{justsome(a)}First, for $\delta_{0}\in S$, look at the different cases in the definition of $p_{\eta,\delta,S}^{\ast}$:

        \begin{enumerate}
            \item For case \ref{maindefa}, $p_{\eta,\delta,S}^{\ast}\restriction\delta_{0}=((\boldsymbol{T}_{\!<\delta})^{[\eta]})\restriction\delta_{0}=(\boldsymbol{T}_{<\delta_{0}})^{[\eta]}=p_{\eta,\delta_{0},S\cap\delta_{0}}^{\ast}$.
            
            \item For case \ref{maindefb}, recalling $\lg(\eta)<\delta_{0}$, the initial segments of $\eta$ are clearly both in $p_{\eta,\delta,S}^{\ast}\restriction\delta_{0}$ and in $p_{\eta,\delta_{0},S\cap\delta_{0}}^{\ast}$; for all $\eta\trianglelefteq\nu\in\boldsymbol{T}_{<\delta_{0}}$ it holds by clause \ref{maindefb(ii)} of the definition that $\nu\in p_{\eta,\delta,S}^{\ast}\Longleftrightarrow\nu\in p_{\eta,\delta_{0},S\cap\delta_{0}}^{\ast}$.
            
            \item For cases \ref{maindefc} and \ref{maindefd}, for each $\nu\in\boldsymbol{T}_{<\delta_{0}}$ the relevant clauses are \ref{maindefc(i)} of \ref{maindefc} and \ref{maindefd(i)} of \ref{maindefd}. Those clauses trivially imply $\nu\in p_{\eta,\delta,S}^{\ast}\Longleftrightarrow\nu\in p_{\eta,\delta_{0},S\cap,\delta_{0}}^{\ast}$.
        \end{enumerate}
        
        \item Next, take any $\delta_{0}\in S_{\ast}\cap\delta\setminus S$:

        \begin{enumerate}
            \item For $\delta_{0}<\sup(S)$, there is $\delta_{0}<\delta'\in S$, by the induction hypothesis $p_{\eta,\delta',S\cap\delta'}^{\ast}\restriction\delta_{0} = p_{\eta,\delta_{0},S\cap\delta_{0}}^{\ast}$. By the clause \ref{justsome(a)} $p_{\eta,\delta,S}^{\ast}\restriction\delta' = p_{\eta,\delta',S\cap\delta'}^{\ast}$ thus $p_{\eta,\delta,S}^{\ast}\restriction\delta_{0} = p_{\eta,\delta_{0},S\cap\delta_{0}}^{\ast}$ follows.
            
            \item For $\delta_{0}\ge\sup(S)$, hence $\forall\delta'\in S$, $\delta_{0}>\delta'$ (else it is in the case of the previous clause) Clearly $p_{\eta,\delta,S}^{\ast}\restriction\delta_{0} \subseteq p_{\eta,\delta_{0},S\cap\delta_{0}}^{\ast}$: to prove the other inclusion let $\nu \in p_{\eta,\delta_{0},S\cap\delta_{0}}^{\ast}$.

            \begin{enumerate}
                \item If for some $\delta'\in S\cap\delta_{0}$, $\lg(\nu)<\delta'$ then by the induction assumption $\nu\in p_{\eta,\delta',S\cap\delta'}^{\ast}$, which by the previous clause implies $\nu\in p_{\eta,\delta,S}^{\ast}$.

                \item If for all $\delta'\in S\cap\delta_{0}$, $\lg(\nu)\geq\delta'$ and $\delta_{1}=\lg(\nu)$ is the last element of $S$: in both relevant cases of the definition (\ref{maindefc}, \ref{maindefd}) the level $\delta_{1}$ of the condition is determined by the previous levels of $S$ and by whether or not it is successful, so $p_{\eta,\delta,S}^{\ast}\cap\boldsymbol{T}_{\delta_{1}} = p_{\eta,\delta_{0},S}^{\ast}\cap\boldsymbol{T}_{\delta_{1}}$;in particular, $\nu\in p_{\eta,\delta,S}^{\ast}$.
                
                \item If for all $\delta'\in S\cap\delta_{0}$ we have $\lg(\nu)\geq\delta'$ and $\lg(\nu)\notin S$: then the level $\lg(\nu)$ is determined entirely by the restrictions to previous levels, so we are done.
            \end{enumerate}
        \end{enumerate}
    \end{enumerate}
    
    \item For all $\nu\in p_{\eta,\delta.S}^{\ast}$, first we will show that $\nu\trianglelefteq\eta$ or $\eta\triangleleft\nu$, the proof splits to cases according to the cases in the definition \ref{trullymaindef}(\ref{maindef}):

    \begin{enumerate}
        \item For case \ref{maindefa} it is clear.
        
        \item For case \ref{maindefb} it is also clear from the definition and the induction hypothesis.
        
        \item For case \ref{maindefc}:

        \begin{enumerate}
            \item For \ref{maindefc(i)} by the induction hypothesis.
            
            \item For \ref{maindefc(ii)} also, by the induction hypothesis, each such $\nu$ has $\eta\triangleleft\nu$.
        \end{enumerate}
        
        \item For case \ref{maindefd}:

        \begin{enumerate}
            \item For $\nu$ chosen in clause \ref{maindefd(i)} we have $\eta\triangleleft\nu$ or $\nu\trianglelefteq\eta$ by the induction hypothesis.
            
            \item For $\nu$ chosen in clause \ref{maindefd(ii)} it holds that $\eta\triangleleft\nu$, again using the induction hypothesis.

            \item For a node $\nu$ chosen in clause \ref{maindefd(iii)}, since $\nu\in\lim_{\delta_{1}}(p_{\eta,\delta_{1},S\cap\delta_{1}}^{\ast})$ and by the induction hypothesis, $\eta\trianglelefteq\nu$.
        \end{enumerate}
    \end{enumerate}

    Second, it remains to prove that $\eta$ is the maximal node for which each other branch is an extension or an initial segment of it.

    In case \ref{maindefa} it is clear; in cases \ref{maindefb} and \ref{maindefc} it follows from the induction hypothesis, the node $\eta$ is the trunk of the condition $p_{\eta,\delta_{1},S\cap\delta_{1}}^{\ast}$ for each $\delta_{1}\in\delta\cap S_{\ast}$ and so it has $\theta_{\epsilon}$ extensions to the level of the trunk +1; those extensions will be in the new condition $p_{\eta,\delta,S}^{\ast}$ thus $\eta$ will be a trunk there as well. 
    For case \ref{maindefd} recall that $\delta_{1}$ is a limit cardinal $>\lg(\eta)$, we can use the induction hypothesis again observing that before the $\delta_{1}$-th level there are no new prunings that didn't exist in $p_{\eta,\delta_{1},S\cap\delta_{1}}^{\ast}$, i.e. $p_{\eta,\delta,S}^{\ast}\restriction\delta_{1}=p_{\eta,\delta_{1},S\cap\delta_{1}}^{\ast}$ therefore if there were a different trunk containing $\eta$, it would have been a trunk of $p_{\eta,\delta_{1},S\cap\delta_{1}}^{\ast}$ as well --- a contradiction.

    \item Using induction, first we will show that $p_{\eta,\delta,S}^{\ast}\in\mathbb{Q}_{\delta}^{0}$, checking the clauses in definition \ref{veryfirstforcing}:

    \begin{enumerate}
        \item For clause \ref{veryfirstforcinga}; $p_{\eta,\delta,S}^{\ast}$ is a tree (follows directly from the induction hypothesis) and it has a trunk $\eta$ by part \ref{properties2} of this claim.

        \item To show clause \ref{veryfirstforcingb}, let $\nu\in p_{\eta,\delta,S}^{\ast}$. Assume $\eta\trianglelefteq\nu$ (the case of $\nu\triangleleft\eta$ follows from the case $\nu=\eta$) to show that there is an extension of $\nu$ to the level $\delta$:

        \begin{enumerate}
            \item In case \ref{maindefa} of definition \ref{trullymaindef}, let $\nu\trianglelefteq\nu'\in\boldsymbol{T}_{\delta}$, then it holds that also $\eta\trianglelefteq\nu'$ thus $\nu'\in\lim_{\delta}(p_{\eta,\delta,S}^{\ast})$.

            \item In case \ref{maindefb} of definition \ref{trullymaindef}, as $S$ is tenuous with no last element, if $\lg(\nu)<\sup(S)$ there is a closed unbounded subset $C$ of $\sup(S)$ with $\min(C)>\lg(\nu)$, such that 
            $$\alpha\in C \Rightarrow [\alpha = \sup(C\cap\alpha) \Leftrightarrow \alpha\notin S].$$

            Let $\langle\alpha_{i} : i < \zeta\rangle$ list $C$ in increasing order.

            We choose $\nu_{i}\in p_{\eta,\alpha_{i+1},S\cap\alpha_{i+1}}^{\ast}$, $\triangleleft$-increasing, $\nu\triangleleft\nu_{i}$ and $\lg(\nu_{i})=\alpha_{i}$ --- this is easy --- and let $\varrho=\underset{i<\zeta}{\bigcup}\nu_{i}$.
            Now if $\lg(\varrho)=\sup(S)=\delta$ then $\varrho\in\lim_{\delta}(p_{\eta,\delta,S}^{\ast})$ and we are done. So assume $\lg(\varrho) = \sup(S) < \delta$; clearly $\varrho\in p_{\eta,\delta,S}^{\ast} \cap \boldsymbol{T}_{\lg(\varrho)}$ hence $(p_{\eta,\delta,S}^{\ast})^{[\varrho]} =(\boldsymbol{T}_{\!<\delta})^{[\varrho]}$ so the derived conclusion is clear. 
            Finally, if $\lg(\nu)\geq\sup(S)$ then every $\nu\trianglelefteq\nu'$ with $\lg(\nu')=\delta$ has the property that for all $\delta_{1}\in S\setminus(\lg(\eta)+1)$, $\zeta<\delta_{1}$ it holds that $\nu'\restriction\zeta = \nu\restriction\zeta\in p_{\eta,\delta_{1},S\cap\delta_{1}}^{\ast}$ and so $\nu'\in\lim_{\delta}(p_{\eta,\delta,S}^{\ast})$.

            \item In case \ref{maindefc} of definition \ref{trullymaindef}, so $\delta_{1}=\max(S)$. First assume that $\nu\in p_{\eta,\delta,S}^{\ast}$ satisfies $\lg(\nu)<\max(S)=\delta_{1}$, then $\nu\in p_{\eta,\delta_{1},S\cap\delta_{1}}^{\ast}$ and by the induction hypothesis, there is some $\nu'\in\lim_{\delta_{1}}(p_{\eta,\delta_{1},S\cap\delta_{1}}^{\ast})$ with $\nu\trianglelefteq\nu'$, by the definition it also holds that $\nu'\in p_{\eta,\delta,S}^{\ast}$. Thus, it is left to prove the claim for any $\nu\in p_{\eta,\delta,S}^{\ast}$ such that $\lg(\nu)\geq\delta_{1}$; clearly for any extension $\nu\trianglelefteq\nu'\in\lim_{\delta}(p_{\eta,\delta,S}^{\ast})$, $\nu'\restriction\delta_{1} \in p_{\eta,\delta,S}^{\ast}$ and so $\nu'\restriction\xi \in p_{\eta,\delta,S}^{\ast}$ for all $\delta_{1}\leq\xi<\delta$.
            
            \item In case \ref{maindefd} of definition \ref{trullymaindef}, so $\nu\in p_{\eta,\delta,S}^{\ast}$, $\delta_{1}=\max(S)$ and $\delta_{1}$ is successful let $\beta=\lg(\nu)$:

            \begin{enumerate}
                \item First assume $\beta<\delta_{1}$, by induction hypothesis there is a node $\nu\trianglelefteq\nu'$, $\nu' \in  \lim_{\delta_{1}}(p_{\eta,\delta_{1},S\cap\delta_{1}}^{\ast})$.
                Now the proof splits to cases:

                \uline{Case 1:} if $\nu'\notin\lim_{\delta_{1}}(r_{\delta_{1}}^{\ast})$, then $\nu'\in p_{\eta,\delta,S}^{\ast}$, hence we have reduced the problem to the case $\beta=\delta_{1}$ dealt with below,

                \uline{Case 2:} if $\nu'\in\lim_{\delta_{1}}(r_{\delta_{1}}^{\ast})$ we still know that $\eta\in\nu'$ hence $\eta\in r_{\delta_{1}}^{\ast}$ hence for some $\varrho\in\Lambda_{\delta_{1}}^{\ast}$ we have $\eta \in q_{\delta_{1},\varrho}^{\ast}$ hence there is $\nu''\in\lim_{\delta_{1}}(q_{\delta_{1},\eta}^{\ast})$ and we continue with $\beta=\delta_{1}$ below.

                \item Second assume $\beta\geq\delta_{1}$, now every possible extension is being chosen after the level of height $\delta_{1}$, so by the previous clause certainly there is an element in the $\beta$ level by clause \ref{maindefd(ii)}.
            \end{enumerate}
        \end{enumerate}
        
        \item In successor levels all the extensions are taken, as defined in $\mathbb{Q}_{\delta}^{0}$.
        
        \item The set $S$ is tenuous and it holds that $S_{p}\subseteq S$ by the next clause, so $S_{p}$ (the set of the levels with the prunes) is also tenuous.
    \end{enumerate}

    Now we can see that $p_{\eta,\delta,S}^{\ast}\in\mathbb{Q}_{\delta}'$:
    \begin{itemize}
        \item Let $\delta'$ be $\lg(\mathrm{tr}(p))<\delta'\in S_{\ast}$; observe that in all the cases of the definition it holds that $p_{\eta,\delta,S}^{\ast}\restriction\delta'=p_{\eta,\delta',S\cap\delta'}^{\ast}\in\mathbb{Q}_{\delta'}$ and so we are done.
    \end{itemize}
    
    \item Looking at the definition, in case \ref{maindefa} trivial; for case \ref{maindefb} we will have that $S_{p_{\eta,\delta,S}^{\ast}}=\underset{\delta'\in S}{\bigcap}S_{p_{\eta,\delta',S\cap\delta'}^{\ast}}$ so by induction $S_{p_{\eta,\delta,S}^{\ast}}\subseteq S$; In case \ref{maindefc}, $S_{p_{\eta,\delta.S}^{\ast}}=S_{p_{\eta,\delta_{1},S\cap\delta_{1}}^{\ast}}$ and in case \ref{maindefd}, $S_{p_{\eta,\delta.S}^{\ast}} = S_{p_{\eta,\delta_{1},S\cap\delta_{1}}^{\ast}}$ or $S_{p_{\eta,\delta.S}^{\ast}} = S_{p_{\eta,\delta_{1},S\cap\delta_{1}}^{\ast}}\cup\{\delta_{1}\}$.
    Using the induction hypothesis and as $\delta_{1}\in S$, we are done.

    \item Reading the definition \ref{trullymaindef}, clearly $\mathbb{Q}_{\delta}'\subseteq\mathbb{Q}_{\delta}^{0}$ and $\mathbb{Q}_{\delta}\subseteq\mathbb{Q}_{\delta}'$ follows by clause \ref{properties3} because $\mathbb{Q}_{\delta} = 
    \{p_{\eta,\delta,S}^{\ast} : \eta\in\boldsymbol{T}_{\!<\delta} \mbox{ and } S = S_{\ast}\cap\delta\mbox{ is tenuous}\}$, so clause \ref{properties5} follows by clause \ref{properties3}.
    So clause \ref{properties5} holds indeed.

    To show $\mathbb{Q}_{\lambda}=\mathbb{Q}_{\lambda}'$, assume by contradiction that there is $p\in\mathbb{Q}_{\lambda}' \setminus \mathbb{Q}_{\lambda}$, so for all $\delta\in S_{\ast}$ we get $p\restriction\delta\in\mathbb{Q}_{\delta}$.
    Let $S = \bigcup\limits_{\delta\in S_{\ast}}S_{p\restriction\delta}$; if $S$ has a last element, then for some $\delta_{\ast}\in S_{\ast}$, $S = S_{p\restriction\delta_{\ast}}$ and so $p = 
    \{\nu \in \boldsymbol{T}_{<\lambda} : \nu\in p\restriction\delta_{\ast} \vee \nu\restriction\delta_{\ast} \in \lim_{\delta_{\ast}}(p)\}$, as $\max(S)<\delta$ and by clauses \ref{maindefc} and \ref{maindefd} of Definition \ref{trullymaindef}(\ref{maindef}), $p\in\mathbb{Q}_{\lambda}$ follows. Otherwise $S$ has no last element; $\nu\in p$ iff for each $\delta\in S$ either $\nu\in p\restriction\delta$ or $\lg(\nu)\geq\delta$ and $\forall\zeta<\delta$, $\nu\restriction\zeta\in p\restriction\delta$ and so by clause \ref{maindefb} of definition \ref{trullymaindef}(\ref{maindef}), $p\in\mathbb{Q}_{\lambda}$.
\end{enumerate}
\end{proof}
\end{claim}

\begin{claim}\label{afterforcingactually}
Let $\delta\in S_{\ast}\cup\{\lambda\}$;
\begin{enumerate}
    \item Let $p,q\in\mathbb{Q}_{\delta}^{0}$, if $p,q$ are compatible then $p\cap q\in\mathbb{Q}_{\delta}^{0}$.

    \item Let $p,q\in\mathbb{Q}_{\delta}'$, then $p,q$ are compatible if and only if $p\cap q\in\mathbb{Q}_{\delta}'$.

    \item Let $p,q\in\mathbb{Q}_{\delta}$, then $p,q$ are compatible if and only if $p\cap q\in\mathbb{Q}_{\delta}$.

    \item \label{whatantichainis1}Let $p,q\in\mathbb{Q}_{\delta}'$, then $p,q$ are compatible if and only if $\mathrm{tr}(p)\in q\wedge \mathrm{tr}(q)\in p$.

    \item \label{whatantichainis2}Let $p,q\in\mathbb{Q}_{\delta}$, then $p,q$ are compatible if and only if $\mathrm{tr}(p)\in q\wedge \mathrm{tr}(q)\in p$.
\end{enumerate}

\begin{proof}
In fact we saw the existence of most of the statements in this claim
already. Observe:
\begin{enumerate}
    \item If $p$ and $q$ are compatible, let $r\in\mathbb{Q}_{\delta}^{0}$
    be such that $r\subseteq p,q$, then $\mathrm{tr}(p),\mathrm{tr}(q)\trianglelefteq \mathrm{tr}(r)$.
    Now, $r\subseteq p\cap q$, assume without loss of generality $\mathrm{tr}(p)\triangleleft \mathrm{tr}(q)=\eta$, then $\eta$ will be the trunk of $p\cap q$. For each $\eta\in p\cap q$, the sets $\{\nu\in\lim_{\delta}(p) : \eta\triangleleft\nu\}$ and $\{\nu\in\lim_{\delta}(q) : \eta\triangleleft\nu\}$ must have a non-empty intersection as $S_{p},S_{q}$ are tenuous.
    For all $\eta\in p\cap q$, 
    $$\{j\in\theta_{\lg(\eta)} : \eta\overset{\frown}{}\langle j\rangle\in p\} = \{j\in\theta_{\lg(\eta)} : \eta\overset{\frown}{}\langle j\rangle\in q\} = \theta_{\lg(\eta)}$$ 
    and so $\{j\in\theta_{\lg(\eta)} : \eta\overset{\frown}{}\langle j\rangle\in p\cap q\}$.
    Finally, as $S_{p},S_{q}$ are tenuous, so is $S_{p\cap q}\subseteq S_{p}\cup S_{q}$ (by claim \ref{stationaryunions}), thus, $p\cap q \in \mathbb{Q}_{\delta}^{0}$.

    \item This clause and the following are shown by simultaneous induction: considering the forcing $\mathbb{Q}_{\delta}'$, if $p$ and $q$ are compatible there exists a condition $r\in\mathbb{Q}_{\delta}'$: $r\subseteq p,q$, thus $\mathrm{tr}(p),\mathrm{tr}(q)\trianglelefteq \mathrm{tr}(r)$. Let $\delta_{1}\in\delta\cap S_{\ast}$, $\lg(\mathrm{tr}(p))<\delta_{1}$; as $p\restriction\delta_{1}, q\restriction\delta_{1}, r\restriction\delta_{1} \in \mathbb{Q}_{\delta_{1}}$ and $r\restriction\delta_{1} \subseteq p\restriction\delta_{1}, q\restriction\delta_{1}$ and the following clause's induction assumption, we conclude that $p\cap q\restriction\delta_{1} \in \mathbb{Q}_{\delta_{1}}$ and $p\cap q\in\mathbb{Q}_{\delta}'$ follows. The other direction is trivial.

    \item We use induction; considering the forcing $\mathbb{Q}_{\delta}$, if $p$ and $q$ are compatible there exists a condition $r\in\mathbb{Q}_{\delta}$: $r\subseteq p,q$, thus $\mathrm{tr}(p),\mathrm{tr}(q)\trianglelefteq \mathrm{tr}(r)$. Assume without loss of generality $\mathrm{tr}(p)\triangleleft \mathrm{tr}(q)=\eta$ and let $S=S_{p}\cup S_{q}$.
    For $\nu\in\boldsymbol{T}_{<\delta'}$ with $\delta'\in S$, $\nu\in p\cap q \Longleftrightarrow \nu\in\lim_{\delta'}(p\restriction\delta')$ and $\nu\in\lim_{\delta'}(q\restriction\delta')$ which by the induction hypothesis implies $\nu\in\lim_{\delta'}(p\cap q\restriction\delta')$. In addition, one of the following holds:

    \begin{enumerate}
        \item $\delta'$ is not successful,

        \item $\delta'$ is successful and $\nu\notin\lim_{\delta'}(r_{\delta'}^{\ast})$,

        \item $\delta'$ is successful and $\nu\in\lim_{\delta'}(r_{\delta'}^{\ast})\cap(\bigcup\{\lim_{\delta'}(q_{\delta',\eta'}^{\ast}):\eta'\in\Lambda_{\delta'}^{\ast}\})$.
    \end{enumerate}

    There are no additional prunings, therefore $p\cap q = p_{\eta,\delta,S}^{\ast}$.

    \item This clause and the next one are shown by simultaneous induction on $\delta$. Considering the forcing $\mathbb{Q}_{\delta}'$:

    \begin{itemize}
        \item For the first direction, assume $p$ and $q$ are compatible; thus there exists a condition $r\in\mathbb{Q}_{\delta}'$: $r\subseteq p,q$. In particular, $\mathrm{tr}(p),\mathrm{tr}(q)\trianglelefteq \mathrm{tr}(r)$ thus $\mathrm{tr}(p),\mathrm{tr}(q)\in r\subseteq p\cap q$.

        \item For the other direction, assume $\mathrm{tr}(p)\in q\wedge \mathrm{tr}(q)\in p$, let $r=p\cap q$ and by previous clause $r\in\mathbb{Q}_{\delta}^{0}$. In particular, $\lg(\mathrm{tr}(p)),\lg(\mathrm{tr}(q))<\delta_{1}$ and since $p,q\in\mathbb{Q}_{\delta}'$ it implies that $p\restriction\delta_{1}, q\restriction\delta_{1} \in \mathbb{Q}_{\delta_{1}}$.
        We can use the induction hypothesis to conclude that $(p\restriction\delta_{1})\cap(q\restriction\delta_{1})=r\restriction\delta_{1}\in\mathbb{Q}_{\delta_{1}}$; therefore indeed $r\in\mathbb{Q}_{\delta}'$.
    \end{itemize}

    \item Considering the forcing $\mathbb{Q}_{\delta}$, assume it holds for $\mathbb{Q}_{\delta_{1}}$ with $\delta_{1}<\delta$:

    \begin{itemize}
        \item For the first direction, assume that $p$ and $q$ are compatible; thus there exists $r\in\mathbb{Q}_{\delta}$: $r\subseteq p,q$. In particular, $\mathrm{tr}(p),\mathrm{tr}(q)\trianglelefteq \mathrm{tr}(r)$ thus $\mathrm{tr}(p),\mathrm{tr}(q)\in r\subseteq p\cap q$.

        \item For the other direction, assume $\mathrm{tr}(p)\in q\wedge \mathrm{tr}(q)\in p$ and remember that for some nodes $\eta_{1},\eta_{2}\in\boldsymbol{T}_{\!<\delta}$ and tenuous sets $S_{1},S_{2}\subseteq S_{\ast}\cap\delta$, the conditions are in fact $p=p_{\eta_{1},\delta,S_{1}}^{\ast}$, $q=p_{\eta_{2},\delta,S_{2}}^{\ast}$. Recall the assumption and assume by symmetry that $\eta_{1}\trianglelefteq\eta_{2}$.
        Let $S=S_{1}\cup S_{2}$ and we will show that $p_{\eta_{2},\delta,S}^{\ast}\subseteq p\cap q$; this is indeed a condition in the forcing $\mathbb{Q}_{\delta}$ looking at definition \ref{trullymaindef}. Let $\nu\in p_{\eta_{2},\delta,S}^{\ast}$; the possibilities by clause \ref{maindef} are:

        \begin{itemize}
            \item If $S$ has no last element;

            \begin{itemize}
                \item If $\nu\trianglelefteq\eta_{2}$ then $\nu\in q$; as $\eta_{2}\in p$ it follows that $\nu\in p\cap q$.

                \item If for some $\delta_{1}\in S$, $\nu\in p_{\eta_{2},\delta_{1},S\cap\delta_{1}}^{\ast}$ then by the induction assumption, $p_{\eta_{2},\delta_{1}, S\cap\delta_{1}}^{\ast} \subseteq p\cap q \cap \boldsymbol{T}_{<\delta_{1}}$ so $\nu\in p\cap q$.

                \item If $\forall\delta_{1}\in S\forall\zeta<\delta_{1}:\nu\restriction\zeta\in p_{\eta_{2},\delta_{1},S\cap\delta_{1}}^{\ast}$, by the induction assumption $\forall\delta_{1}\in S\forall\zeta<\delta_{1}:\nu\restriction\zeta\in p\cap q$. If $S_{1}$ or $S_{2}$ had a last element, it was below $\sup(S)$ and in all the construction possibilities it can be seen that this implies $\nu\in p$ and $\nu\in q$.
            \end{itemize}

            \item If $S$ has a last element $\delta_{1}$, which is not successful;

            \begin{itemize}
                \item If $\nu\in p_{\eta_{2}, \delta_{1}, S\cap\delta_{1}}^{\ast}$, by the induction assumption $p_{\eta_{2},\delta_{1},S\cap\delta_{1}}^{\ast}\subseteq p\cap q\cap\boldsymbol{T}_{<\delta_{1}}$ so $\nu\in p\cap q$.

                \item If $\nu\restriction\delta_{1}\in\lim_{\delta_{1}}(p_{\eta_{2},\delta_{1},S\cap\delta_{1}}^{\ast})$, by the induction assumption $p_{\eta_{2},\delta_{1},S\cap\delta_{1}}^{\ast} \subseteq p\cap q\cap\boldsymbol{T}_{<\delta_{1}}$, so $\nu\restriction\delta_{1}\in p\cap q$. For each one of $S_{1},S_{2}$, if it doesn't contain $\delta_{1}$ then $\nu$ belongs to the matching condition ($p$ or $q$), while if it does contain $\delta_{1}$, as $\delta_{1}$ is not successful, the matching condition, say $p$, will have that $\nu\restriction\delta_{1}\in\lim_{\delta_{1}}(p)\Rightarrow\nu\in p$. 
            \end{itemize}
            
            \item If $S$ has a last element $\delta_{1}$, which is successful;

            \begin{itemize}
                \item If $\nu\in p_{\eta_{2},\delta_{1},S\cap\delta_{1}}^{\ast}$, by the
induction assumption $p_{\eta_{2},\delta_{1},S\cap\delta_{1}}^{\ast}\subseteq p\cap q\cap\boldsymbol{T}_{<\delta_{1}}$
so $\nu\in p\cap q$.

                \item If $\nu\in\lim_{\delta_{1}}(p_{\eta_{2},\delta_{1},S\cap\delta_{1}}^{\ast})$
then by the induction assumption $\nu\in\lim_{\delta_{1}}(p\cap q\cap\boldsymbol{T}_{<\delta_{1}})$:

                \begin{itemize}
                    \item In case $\nu\notin\lim_{\delta_{1}}(r_{\delta_{1}}^{\ast})$, for each one of $p,q$, if $S_{1}$ or $S_{2}$ have $\delta_{1}$ as their last element, $\nu\in p$ or $\nu\in q$ accordingly. Else the corresponding $S_{1}$ or $S_{2}$ has all its elements below $\delta_{1}$ and so by the possibilities in Definition \ref{trullymaindef}(\ref{maindef}), $\nu\in p\cap q$.

                    \item Else $\nu\in\lim_{\delta_{1}}(r_{\delta_{1}}^{\ast})$, then
                     $$(*) \ \nu\in\bigcup\{\lim_{\delta_{1}}(q_{\delta_{1},\eta'}^{\ast}) : \eta'\in\Lambda_{\delta_{1}}^{\ast}\}.$$
                    For each of $p,q$, if $S_{1}$ or $S_{2}$ have $\delta_{1}$ as their last element then since $\nu\in\lim_{\delta_{1}}(p\cap q)$ and by $({*})$ $\nu$ is contained in the corresponding condition. Else the corresponding $S_{1}$ or $S_{2}$ has all its elements below $\delta_{1}$ and so by all the possibilities in Definition \ref{trullymaindef}(\ref{maindef}), $\nu\in p\cap q$.
                \end{itemize}

                \item If $\nu\restriction\delta_{1}\in p_{\eta_{2},\delta,S}^{\ast}\cap\boldsymbol{T}_{\delta_{1}}$,since $p_{\eta_{2},\delta,S}^{\ast}\cap\boldsymbol{T}_{\delta_{1}}\subseteq p\cap q\cap\boldsymbol{T}_{\delta_{1}}$, $\nu\restriction\delta_{1}\in p\cap q$ and so $\nu\in p\cap q$, for any possibility for the construction of $p$ and $q$.
            \end{itemize}
        \end{itemize}
    \end{itemize}
\end{enumerate}
\end{proof}
\end{claim}

Recall the required properties of the forcings discussed in Remark
\ref{beforeforcing}, the first was shown in \ref{someproperties}(\ref{properties5})
and the rest are proven below:
\begin{claim}
Let $\delta\in S_{\ast}\cup\{\lambda\}$;
\begin{enumerate}
\item \label{enu:afterforcingactually2}For a condition $p=p_{\eta,\delta,S}^{\ast}\in\mathbb{Q}_{\delta}$
and a node $\nu\in p$ , we have $p\leq_{\mathbb{Q}_{\delta}}p^{[\nu]}\in\mathbb{Q}_{\delta}$
and $\mathrm{tr}(p^{[\nu]})=\max\{\mathrm{tr}(p),\nu\}$, if $\eta\triangleleft\nu$
then also $p^{[\nu]}=p_{\nu,\delta,S}^{\ast}$ holds, the same is
true for $\mathbb{Q}_{\delta}'\subseteq\mathbb{Q}_{\delta}^{0}$ .
\item It holds that $\boldsymbol{T}_{\!<\delta}\in\mathbb{Q}_{\delta}$ and
$\boldsymbol{T}_{\!<\delta}\in\mathbb{Q}_{\delta}'$, therefore $\boldsymbol{T}_{\!<\delta}$
is the minimal condition of $\mathbb{Q}_{\delta}$ and $\mathbb{Q}_{\delta}'$.
\end{enumerate}
\end{claim}

\begin{proof}
For $\delta\in S_{\ast}\cup\{\lambda\}$:
\begin{enumerate}
\item Assume $p=p_{\eta,\delta,S}^{\ast}\in\mathbb{Q}_{\delta}$ and let
$\nu\in p$.

\begin{itemize}
\item If $\nu\trianglelefteq\eta$ then $p^{[\nu]}=p\in\mathbb{Q}_{\delta}$;
in particular $\mathrm{tr}(p^{[\nu]})=\eta$.
\item Else, $\eta\triangleleft\nu$. In that case $p^{[\nu]}=p_{\nu,\delta,S}^{\ast}$,
we will show that using induction, looking at the clauses of definition
\ref{trullymaindef}(\ref{maindef}):

\begin{enumerate}
\item If, as in case \ref{maindefa}, $p=\boldsymbol{T}_{\!<\delta}^{[\eta]}$
then $p^{[\nu]}=\boldsymbol{T}_{\delta}^{[\nu]}$ which is in fact
$p_{\nu,\delta,\varnothing}^{\ast}$ and thus belongs to $\mathbb{Q}_{\delta}$
and $\mathrm{tr}(p^{[\nu]})=\nu$.
\item If, as in case \ref{maindefb}, there is a $\lg(\eta)<\delta'\in S$
with $\nu\in p_{\eta,\delta',S\cap\delta'}^{\ast}$, then by the induction
hypothesis $p^{[\nu]}\cap\boldsymbol{T}_{<\delta'}\in\mathbb{Q}_{\delta'}$.
In addition, $p^{[\nu]}=p_{\nu,\delta,S}^{\ast}$ and therefore belongs
to $\mathbb{Q}_{\delta}$. However $(\forall\eta_{1},\eta_{2}\in\boldsymbol{T}_{\!<\delta})(\eta_{1}\triangleleft\eta_{2}\in p_{\eta_{1},\delta,S}^{\ast}\Rightarrow p_{\eta_{1},\delta,S}^{\ast}\leq_{\mathbb{Q}_{\delta}}p_{\eta_{2},\delta,S}^{\ast})$
hence we are done. In the case $\lg(\nu)\geq\sup(S)$, $p^{[\nu]}=\boldsymbol{T}_{\!<\delta}^{[\nu]}=p_{\nu,\delta,S}^{\ast}$.
\item If, as in cases \ref{maindefc} and \ref{maindefd}, $S$ has
a last element $\delta_{1}<\delta$, such that $\lg(\eta)<\delta_{1}\in S$,
then if $\delta_{1}\leq\lg(\nu)$, $p^{[\nu]}=\boldsymbol{T}_{\!<\delta}^{[\nu]}=p_{\nu,\delta,S}^{\ast}$. 

\begin{enumerate}
\item In case \ref{maindefc} ($\delta_{1}$ is not successful):

\begin{enumerate}
\item If $\lg(\nu)<\delta_{1}$, then $p^{[\nu]}$ contains all the nodes
of the shape $\nu'\in p_{\eta,\delta,S}^{\ast}$ such that: (1) $\nu'\triangleleft\nu$,
(2) $\nu\trianglelefteq\nu'$ and $\lg(\nu')<\delta_{1}$ and $\nu'\in p_{\eta,\delta_{1},S\cap\delta_{1}}^{\ast}$
or (3) $\nu\trianglelefteq\nu'$ and $\lg(\nu')\geq\delta_{1}$ and
$\nu'\restriction\delta_{1}\in\lim_{\delta_{1}}(p_{\eta,\delta_{1},S\cap\delta_{1}}^{\ast})$.
By induction we have that $p_{\eta,\delta_{1},S\cap\delta_{1}}^{\ast[\nu]}=p_{\nu,\delta_{1},S\cap\delta_{!}}^{\ast}$
therefore $p^{[\nu]}=p_{\nu,\delta,S}^{\ast}\in\mathbb{Q}_{\delta}$
and $\mathrm{tr}(p^{[\nu]})=\nu$.
\item If $\lg(\nu)\geq\delta_{1}$, then $p^{[\nu]}$ contains all the nodes
of the shape $\nu'\in p_{\eta,\delta,S}^{\ast}$ such that: (1) $\nu'\triangleleft\nu$
or (2) $\nu\trianglelefteq\nu'$ (then $\lg(\nu')\geq\delta_{1}$)
and $\nu'\restriction\delta_{1}=\nu\restriction\delta_{1}\in\lim_{\delta_{1}}(p_{\eta,\delta_{1},S\cap\delta_{1}}^{\ast})$,
so in fact any $\nu'$ with $\nu\trianglelefteq\nu'$ is in that group.
Clearly $p^{[\nu]}=p_{\nu,\delta,S}^{\ast}\in\mathbb{Q}_{\delta}$
and $\mathrm{tr}(p^{[\nu]})=\nu$.
\end{enumerate}
\item In case \ref{maindefd} ($\delta_{1}$ is successful)

\begin{enumerate}
\item If $\lg(\nu)<\delta_{1}$, then $p^{[\nu]}$ contains all the nodes
of the shape $\nu'\in p_{\eta,\delta,S}^{\ast}$ such that (1) $\nu'\triangleleft\nu$,
(2) $\nu\trianglelefteq\nu'$ and $\lg(\nu')<\delta_{1}$ and $\nu'\in p_{\eta,\delta_{1},S\cap\delta_{1}}^{\ast}$,
(3) $\nu\triangleleft\nu'$ and $\lg(\nu')=\delta_{1}$ and ($\nu'\in\lim_{\delta_{1}}(p_{\nu,\delta_{1},S\cap\delta_{1}}^{\ast})\cap\lim_{\delta_{1}}(r_{\delta_{1}}^{\ast})\cap(\underset{\eta'\in\Lambda_{\delta_{1}}^{\ast}}{\bigcup}\lim_{\delta_{1}}(q_{\delta_{1},\eta'}^{\ast}))$
or $\nu'\in\lim_{\delta_{1}}(p_{\nu,\delta_{1},S\cap\delta_{1}}^{\ast})\setminus\lim_{\delta_{1}}(r_{\delta_{1}}^{\ast})$)
or (4) $\nu\triangleleft\nu'$, $\lg(\nu')>\delta_{1}$ and $\nu'\restriction\delta_{1}\in p_{\eta,\delta,S}^{\ast}\cap\boldsymbol{T}_{\delta_{1}}$.
By induction, $p_{\eta,\delta_{1},S\cap\delta_{1}}^{\ast[\nu]}=p_{\nu,\delta_{1},S\cap\delta_{1}}^{\ast}$
and since $r_{\delta_{1}}^{\ast}$ and $\forall\eta'\in\Lambda_{\delta_{1}}^{\ast}:q_{\delta_{1},\eta'}^{\ast}$
do not depend on the trunk, we see that $p_{\eta,\delta,S}^{\ast[\nu]}\restriction\delta_{1}=p_{\nu,\delta,S}^{\ast}\restriction\delta_{1}$
and $p_{\eta,\delta,S}^{\ast[\nu]}\cap\boldsymbol{T}_{\delta_{1}}=p_{\nu,\delta,S}^{\ast}\cap\boldsymbol{T}_{\delta_{1}}$,
thus the equality follows also for $\nu\triangleleft\nu'$ such that
$\lg(\nu')>\delta_{1}$ and $p_{\eta,\delta,S}^{\ast[\nu]}=p_{\nu,\delta,S}^{\ast}\in\mathbb{Q}_{\delta}$
and $\mathrm{tr}(p^{[\nu]})=\nu$.
\item If $\lg(\nu)\geq\delta_{1}$ then $p^{[\nu]}$ contains all the nodes
of the shape $\nu'\in p_{\eta,\delta,S}^{\ast}$ such that (1) $\nu'\trianglelefteq\nu$
or (2) $\nu\triangleleft\nu'$, $\lg(\nu')>\delta_{1}$ and $\nu'\restriction\delta_{1}\in p_{\eta,\delta,S}^{\ast}\cap\boldsymbol{T}_{\delta_{1}}$,
which since $\nu'\restriction\delta_{1}=\nu\restriction\delta_{1}$
implies $p_{\eta,\delta,S}^{\ast[\nu]}=p_{\nu,\delta,S}^{\ast}\in\mathbb{Q}_{\delta}$
and $\mathrm{tr}(p^{[\nu]})=\nu$.
\end{enumerate}
\end{enumerate}
\end{enumerate}
\end{itemize}

In particular it holds that $\mathrm{tr}(p^{[\nu]})=\max\{\eta,\nu\}$.

We have finished showing that $\nu\in p\in\mathbb{Q}_{\delta}\Rightarrow p^{[\nu]}\in\mathbb{Q}_{\delta}$.
What about $\mathbb{Q}_{\delta}'$?

Let $p\in\mathbb{Q}_{\delta}'$, then for each $\delta'\in\delta\cap S_{\ast}$
it holds that $p\restriction\delta'\in\mathbb{Q}_{\delta'}$. Next,
observe that $q=p^{[\nu]}$ for $\eta\trianglelefteq\nu\in p$; then
for all $\lg(\nu)\leq\delta'\in\delta\cap S_{\ast}$, $q\restriction\delta'=(p\restriction\delta')^{[\nu]}$.
Observe that $p\restriction\delta'\in\mathbb{Q}_{\delta'}$ and by
the first part of this clause also $(p\restriction\delta')^{[\nu]}\in\mathbb{Q}_{\delta'}$.
For $\nu\trianglelefteq\eta$, it holds that $p^{[\nu]}=p\in\mathbb{Q}_{\delta}'$
and in particular $\mathrm{tr}(p^{[\nu]})=\eta$; so indeed $\mathrm{tr}(p^{[\nu]})=\max\{\eta,\nu\}$.

By the definition of $p^{[\nu]}$, $p^{[\nu]}\subseteq p$ and since
the order of both forcing $\mathbb{Q}_{\delta}$ and $\mathbb{Q}_{\delta}'$
is inverse inclusion and by what we just showed if $p\in\mathbb{Q}_{\delta}$,
then $p\leq_{\mathbb{Q}_{\delta}}p^{[\nu]}$ and if $p\in\mathbb{Q}_{\delta}'$
then $p\leq_{\mathbb{Q}_{\delta}'}p^{[\nu]}$.

\item It holds that $\boldsymbol{T}_{\!<\delta} = p_{\langle\ \rangle,\delta,\varnothing}^{\ast}$
so trivially it belongs to $\mathbb{Q}_{\delta}$, it then by the
first clause of this claim it follows that $\boldsymbol{T}_{\!<\delta}\in\mathbb{Q}_{\delta}'$
and we are done.
\end{enumerate}
\end{proof}

\begin{lem}
If $p\in\mathbb{Q}_{\delta}$ and $\lg(\mathrm{tr}(p))<\alpha<\delta$ then
$\{p^{[\eta]}:\eta\in p\cap\boldsymbol{T}_{\alpha}\}$ is a maximal
antichain of $\mathbb{Q}_{\delta}$ above $p$, the same holds for
$\mathbb{Q}_{\delta}'$.\end{lem}
\begin{proof}
Let $\eta,\nu\in p\cap\boldsymbol{T}_{\alpha}$ be different, then
$\eta\notin p^{[\nu]}$ and $\nu\notin p^{[\eta]}$; recalling clause
\ref{whatantichainis2} of claim \ref{afterforcingactually} it
follows that $p^{[\eta]},p^{[\nu]}$ are incompatible and so the set
$\{p^{[\eta]}:\eta\in p\cap\boldsymbol{T}_{\alpha}\}$ is an antichain
in $\mathbb{Q}_{\delta}$ above $p$. In addition, let $p\leq_{\mathbb{Q}_{\delta}}q\in\mathbb{Q}_{\delta}$
and let $\eta_{0}\in q\cap\boldsymbol{T}_{\alpha}\subseteq p\cap\boldsymbol{T}_{\alpha}$;
then $p^{[\eta_{0}]}$ is compatible with $q$: their common upper
bound is $q^{[\eta_{0}]}$ and this is in $\mathbb{Q}_{\delta}$ by
what we just showed. Clearly $p^{[\eta_{0}]}\in\{p^{[\eta]}:\eta\in p\cap\boldsymbol{T}_{\alpha}\}$
so this set is indeed a maximal antichain. The proof for $\mathbb{Q}_{\delta}'$
is identical.\end{proof}
\begin{cor}
\label{antichain}Let $\delta\in S_{\ast}$, if $\delta$ is successful
\uline{then} the set $\bar{q}_{\delta}^{\ast}=\langle q_{\delta,\eta}^{\ast}:\eta\in\Lambda_{\delta}^{\ast}\rangle$
is an antichain of $\mathbb{Q}_{\delta}'$ above $r_{\delta}^{\ast}$.\end{cor}
\begin{proof}
Recall that if $\eta\neq\nu\in\Lambda_{\delta}^{\ast}$ \uline{then}
$\eta\notin q_{\delta,\nu}^{\ast}\vee\nu\notin q_{\delta,\eta}^{\ast}$
(see definition \ref{Chi}) and recall \ref{afterforcingactually}(\ref{whatantichainis1}).
\end{proof}

\subsection{Properties of the Forcing}
\begin{claim}\label{completenessfornoninaccessibles}
Let $\delta\in S_{\ast}$
be such that $S_{\ast}\cap\delta$ is non-stationary in $\delta$,
\uline{then} the forcing $\mathbb{Q}_{\delta}$ is strategically
complete in $\mathrm{cf}(\delta)$.

\begin{rem}
Remember that if $\delta\in S_{\ast}$ is not inaccessible then $S_{\ast}\cap\delta$ is not stationary in $\delta$. Also, if $\alpha < \lambda$ and $\delta=\min(S_{\ast}\setminus(\alpha+1))$
then $S_{\ast}\cap\delta$ is not stationary.
\end{rem}

\begin{proof}
First, assume that this holds for each $\delta_{0}<\delta$. Now,
there is a club $E$ of $\delta$ such that $E\cap S_{\ast}=\varnothing$.
Let $p\in\mathbb{Q}_{\delta}$ and $\alpha=\mathrm{cf}(\delta)$, we shall
play the game $\Game_{\alpha}(p,\mathbb{Q}_{\delta})$, determining
a strategy for COM;
\begin{enumerate}
    \item At the first step player COM will choose a condition $p_{0}\geq p$
and after INC chose $q_{0}$, COM chooses a club $E_{0}$ of $\delta$
disjoint to $S_{q_{0}}$.

    \item In successor step $i+1<\alpha$: look at the condition $q_{i}$ that
player INC chose in the $i$-th step; let $\beta_{i}=\lg(\mathrm{tr}(q_{i}))$.
In addition let $\gamma_{i}=\min(E\setminus(\beta_{i}+1))$. Now choose
some $\eta_{i+1}\in q_{i}\cap\boldsymbol{T}_{\gamma_{i}}$; \uline{player
COM will choose} $p_{i+1}=(q_{i})^{[\eta_{i+1}]}$, this is a condition
of the forcing $\mathbb{Q}_{\delta}$ by claim \ref{afterforcingactually}.
Observe that $\mathrm{tr}(q_{i})\trianglelefteq\eta_{i+1}$, $q_{i}\leq_{\mathbb{Q}_{\delta}}p_{i+1}$,
and by the choice made by COM, INC was forced to have $\eta_{i+1}\trianglelefteq \mathrm{tr}(q_{i+1})$. 

    Finally, after INC's $(i+1)^{\mathrm{th}}$ turn, COM will let $E_{i+1}$ be a club contained in $E_i \setminus S_{q_i+1}$: this is possible as $E_i$ is a club of delta and $S_{q_i+1}$ is tenuous.

    \item In limit step $i(\ast)<\delta$: \uline{player COM will choose} $p_{i(\ast)}=\bigcap\limits_{i<i(*)}q_{i}$. Let $S_{i(\ast)} = \bigcup\limits_{i<i(\ast)} S_{q_{i}} \setminus \lg(\nu_{i(\ast)})$ and $\nu_{i(\ast)} = \bigcup\limits_{i<i(\ast)}\mathrm{tr}(q_{i})$.

    \begin{enumerate}
        \item The node $\nu_{i(\ast)}$ belongs to all the conditions that player INC had chosen in the steps $i<i(\ast)$: observe that $$\delta' = \sup\{\beta_{i} : i < i(\ast)\} = \sup\{\gamma_{i} : i < i(\ast)\}$$
        but $\gamma_{i}\in E$ hence $\delta'\in E$. Since $E$ is a club disjoint to $S_{i}$ there is no pruning in the level $\delta'$; in particular, $\nu_{i(\ast)}$ is not being pruned. Thus $\nu_{i(\ast)}\in q_{i}$ for all $i<i(\ast)$.

        \item It remains to show that $p_{i(\ast)}$ is indeed a condition in the forcing and in fact $p_{i(\ast)} = p_{\nu_{i(\ast)},\delta,S_{i(\ast)}}^{\ast}$. First observe that $\mathrm{cf}(\delta')=\mathrm{cf}(i(\ast))$, next:

        \begin{enumerate}
            \item For each node $\nu'\in p_{i(\ast)}$ such that $\lg(\nu')<\lg(\nu_{i(\ast)})$ there is $i<i(\ast)$ such that $\lg(\nu')<\lg(\mathrm{tr}(q_{i}))$ and as
$p_{i(\ast)}$ is the intersection, we get $\nu'\triangleleft \mathrm{tr}(q_{i})$
and so $\nu' \trianglelefteq \bigcup\limits_{i<i(\ast)}\mathrm{tr}(q_{i})$.
Thus $\bigcup\limits_{i<i(\ast)}\mathrm{tr}(q_{i})$ is a node such that there
is no splitting before it in $p_{i(\ast)}$. However in each level
above this node there are splittings as those splittings exist for
each $q_{i}$, and they are ``full'', see definition \ref{veryfirstforcing}(\ref{veryfirstforcing1})(\ref{veryfirstforcingc})
and definition \ref{trullymaindef}. In addition, for each $i<j<i(\ast)$
any splitting in the tree $q_{j}$ exists in the tree $q_{i}$ as
well: this is an increasing sequence of conditions and $q_{j}\subseteq q_{i}$.
It follows that $p_{i(\ast)}$ is a tree with trunk $\nu_{i(\ast)}$
(if we use the filters $D_{\epsilon}$ for $\epsilon<\lambda$, this
is somewhat more delicate, still OK).
\item The set $S_{i(\ast)}$ is tenuous: the set $S_{\ast}\cap\delta$ is
non-stationary by the claim assumption, if $\epsilon\leq\lg(\nu_{i(\ast)})$
then $S_{i(\ast)}\cap\epsilon=\varnothing$ so non- stationary. For
all $\epsilon\in(\lg(\nu_{i(\ast)}),\delta)$, if $S_{\ast}$ doesn't
reflect in $\epsilon$ then $S_{i(\ast)}\restriction\epsilon\subseteq S_{\ast}$
is non- stationary in $\epsilon$ by claim \ref{stationaryunions}(\ref{unionnonstat});
if $S_{\ast}$ reflects to $\epsilon$ then $\epsilon$ is inaccessible
and thus $S_{i(\ast)}\restriction\epsilon$ is a union of $i(\ast)$
sets, non- stationary in $\epsilon$ so by claim \ref{stationaryunions}(\ref{unionnonstat})
$S_{i(\ast)}\restriction\epsilon$ is non- stationary. Putting together
$S_{i(\ast)}$ is tenuous. 
\end{enumerate}
\end{enumerate}

For all $i<i(\ast)$, we shall see that $p_{\nu_{i(\ast)},\delta,S_{i(\ast)}}^{\ast}\subseteq q_{i}$;
as $\mathrm{tr}(q_{i})\trianglelefteq\nu_{i(\ast)}$, observe that $q_{i}^{[\nu_{i(\ast)}]}\subseteq q_{i}$
by a previous lemma. Also, $q_{i}^{[\nu_{i(\ast)}]}$ and $p_{\nu_{i(\ast)},\delta,S_{i(\ast)}}^{\ast}$
have the same trunk, with the first having a smaller stationary set:
$S_{q_{i}^{[\nu_{i(\ast)}]}}\subseteq S_{i(\ast)}$, recalling \ref{afterforcingactually}(\ref{whatantichainis2}),
get $p_{\nu_{i(\ast)},\delta,S_{i(\ast)}}^{\ast}\subseteq q_{i}^{[\nu_{i(\ast)}]}$,
thus $p_{\nu_{i(\ast)},\delta,S_{i(\ast)}}^{\ast}\subseteq q_{i}$
and $p_{\nu_{i(\ast)},\delta,S_{i(\ast)}}^{\ast}\subseteq p_{i(\ast)}$.

We need to also see that $p_{i(\ast)}\subseteq p_{\nu_{i(\ast)},\delta,S_{i(\ast)}}^{\ast}$;
assume this doesn't hold. Then, for some $\nu'\in p_{i(\ast)}$, $\nu'\notin p_{\nu_{i(\ast)},\delta,S_{i(\ast)}}^{\ast}$;
let $\delta'$ be the minimal such that $\nu'\restriction\delta'\notin p_{\nu_{i(\ast)},\delta,S_{i(\ast)}}^{\ast}$;
remembering definition \ref{maindef}, necessarily $\nu'\restriction\delta'\in\lim_{\delta'}(p_{\nu_{i(\ast)},\delta',S_{i(\ast)}\cap\delta'}^{\ast})$
and $\nu'\restriction\delta'\in\lim_{\delta'}(r_{\delta'}^{\ast})\setminus(\bigcup\{\lim_{\delta'}(q_{\delta',\eta'}^{\ast}):\eta'\in\Lambda_{\delta'}^{\ast}\})$.
As $\delta'\in S_{i(\ast)}$, there exists $i<i(\ast)$ such that
$\delta'\in S_{q_{i}}$; since for all $\delta''<\delta'$, $\nu'\restriction\delta''\in p_{\nu_{i(\ast)},\delta,S_{i(\ast)}}^{\ast}\subseteq q_{i}$
so $\nu'\restriction\delta'\in\lim_{\delta'}(q_{i})$ and by the construction
of $q_{i}$ as $\nu'\restriction\delta'\in\lim_{\delta'}(r_{\delta'}^{\ast})\setminus(\bigcup\{\lim_{\delta'}(q_{\delta',\eta'}^{\ast}):\eta'\in\Lambda_{\delta'}^{\ast}\})$
it follows that $\nu'\restriction\delta'\notin q_{i}$, a contradiction
to the assumption $\nu'\in p_{i(\ast)}\subseteq q_{i}$.

We showed that $p_{\nu_{i(\ast)},\delta,S_{i(\ast)}}^{\ast}=p_{i(\ast)}$,
in addition for all $i<i(\ast)$, $q_{i}\leq_{\mathbb{Q}_{\lambda}}p_{i(\ast)}$
so easily $p_{i(\ast)}$ is the smallest supremum of those conditions.

\end{enumerate}
\end{proof}
\end{claim}

\begin{thm}
\label{completenessforinaccessibles}If $\delta\in S_{\ast}\cup\{\lambda\}$
is inaccessible, \uline{then} the forcing $\mathbb{Q}_{\delta}$
is strategically complete in $\delta$.

\begin{proof}
For $\delta\in S_{\ast}\cup\{\lambda\}$ and $p\in\mathbb{Q}_{\delta}$,
we shall play the game $\Game_{\delta}(p,\mathbb{Q}_{\delta})$. We
shall construct inductively the sequence $\langle p_{i},q_{i},E_{i}:i<\delta\rangle$
where $p_{i}$ is the $i$-th move of player COM, $q_{i}$ is the
$i$-th move of player INC, $E_{i}$ is a club in $\delta$ chosen
by COM after INC plays his $i$-th move; it shall be disjoint to $S_{q_{i}}$;
assume that for all $i'<j'<i$: $E_{j'}\subseteq E_{i'}$, this will
prove the desired condition.
\begin{enumerate}
    \item At the first step player COM will choose a condition $p_{0}\geq p$
    and after INC chose $q_{0}$, COM chooses a club $E_{0}$ of $\delta$ disjoint to $S_{q_{0}}$.

    \item In successor step $i + 1 < \delta$: look at the condition $q_{i}$ that player INC chose in the $i$-th step; $E_{i}$ is a club disjoint to $S_{q_{i}}$ s.t. $E_{i}\subseteq\underset{j<i}{\bigcap}E_{j}$ (this club was defined in step $i$); let $\beta_{i}=\lg(\mathrm{tr}(q_{i}))$ and let $\gamma_{i}=\min(E_{i}\setminus(\beta_{i}+1))$. Next, for some node $\eta_{i+1}\in q_{i}\cap\boldsymbol{T}_{\gamma_{i}}$; \uline{player COM will choose} $p_{i+1}=(q_{i})^{[\eta_{i+1}]}$, this is a condition of the forcing $\mathbb{Q}_{\delta}$ by Claim \ref{afterforcingactually}. Observe that $\mathrm{tr}(q_{i})\trianglelefteq\eta_{i+1}$, $q_{i}\leq_{\mathbb{Q}_{\delta}}p_{i+1}$ and by the choice player COM made, she forced player INC to have $\eta_{i+1}\trianglelefteq \mathrm{tr}(q_{i+1})$. Finally, after INC will play his $i+1$-th turn, COM will let $E_{i+1}$ be a club: $E_{i+1}\subseteq E_{i}\setminus S_{q_{i+1}}$; this is possible as $E_{i}$ is a club of $\delta$, $S_{q_{i+1}}$ is tenuous.

    \item In limit step $i(\ast)<\alpha$: \uline{player COM will choose} $p_{i(\ast)} = \bigcap\limits_{i<i(*)}q_{i}$, let $S_{i(\ast)} = \bigcup\limits_{i<i(\ast)} S_{q_{i}} \setminus \lg(\nu_{i(\ast)})$ where $\nu_{i(\ast)} = \bigcup\limits_{i<i(\ast)} \mathrm{tr}(q_{i})$ and $E_{i(\ast)} = \bigcap\limits_{i<i(*)}E_{i}$. Observe that $\lg(\nu_{i(\ast)})<\delta$ since $\delta$ is inaccessible, in addition $E_{i(\ast)}$ is a club in $\delta$ as an intersection of $i(\ast)<\delta$ clubs.

    \begin{enumerate}
        \item The node $\nu_{i(\ast)}$ belongs to all the conditions that player INC had chosen in the steps $i<i(\ast)$: observe that 
        $$\delta'=\sup\{\beta_{i}:i<i(\ast)\}=\sup\{\gamma_{i}:i<i(\ast)\}$$
        as $\mathrm{tr}(q_i) \in q_i$ is $\lhd$-increasing and $q_i$ is decreasing for $i < i(*)$. Clearly $$\{\nu_{i(*)} \restriction \beta : \beta < \delta'\} = \big\{\mathrm{tr}(q_i) \restriction \beta : i < i(*),\ \beta < \lg(\mathrm{tr}(q_i)) \big\} \subseteq \bigcap \{q_i : i < i(*)\}.$$
        
        Since $E_{i(\ast)}$ is a club that is a decreasing intersection of clubs, note $i<i(\ast) \Rightarrow E_{i}\cap S_{q_{i}}=\varnothing\Rightarrow E_{i(\ast)}\cap S_{q_{i}}=\varnothing$, there are no prunings in the level $\delta'$, in particular $\nu_{i(\ast)}$ is not being pruned. Thus $\nu_{i(\ast)}\in q_{i}$ for all $i<i(\ast)$.

        \item It remains to show that $p_{i(\ast)}$ is indeed a condition in the forcing, first observe that $\mathrm{cf}(\delta')=\mathrm{cf}(i(\ast))$ and $\delta'\geq i(\ast)$, next:

        \begin{enumerate}
            \item For each node $\nu'\in p_{i(\ast)}$ such that $\lg(\nu')<\delta'$ there is $i<i(\ast)$ such that $\lg(\nu')<\lg(\mathrm{tr}(q_{i}))$ and as $p_{i(\ast)}$ is the intersection, we get $\nu'\triangleleft \mathrm{tr}(q_{i})$ and so $\nu'\trianglelefteq\bigcup\limits_{i<i(\ast)}\mathrm{tr}(q_{i})$.
            We get that $\bigcup\limits_{i<i(\ast)}\mathrm{tr}(q_{i})$ is a node such that there is no splitting before it in $p_{i(\ast)}$. However, in each level above it (in the $\unlhd$ sense) there are splittings as there are such splittings for each $q_{i}$. In addition, for each $i<j<i(\ast)$ any splitting in the tree $q_{j}$ exists in the tree $q_{i}$ as well: this is an increasing sequence of conditions and $q_{j}\subseteq q_{i}$.
            It follows that $p_{i(\ast)}$ is a tree with trunk $\nu_{i(\ast)}$.
            
            \item The set $S_{i(\ast)}$ is tenuous: as a union of $i(\ast)<\delta=\mathrm{cf}(\delta)$ non-stationary sets, $S_{i(\ast)}$ is non-stationary in $\delta$ by \ref{stationaryunions}(\ref{unionnonstat}). For all $\epsilon<\delta$, if $\epsilon<\delta'$ then $S_{i(\ast)}\cap\epsilon=\varnothing$ so this is trivial hence assume $\epsilon>\delta'$; hence $\epsilon>i(\ast)$; if $S_{\ast}$ doesn't reflect to $\epsilon$ then $S_{i(\ast)}\restriction\epsilon \subseteq S_{\ast}$ is non-stationary in $\epsilon$ by \ref{stationaryunions}(\ref{unionnonstat}); if $S_{\ast}$ reflects to $\epsilon$ then $\epsilon$ is inaccessible and thus $S_{i(\ast)}\restriction\epsilon$ is a union of $i(\ast)$ sets, non-stationary in $\epsilon$ and recall $\epsilon>i(\ast)$ so by \ref{stationaryunions}(\ref{unionnonstat}) $S_{i(\ast)}\restriction\epsilon$ is non-stationary; together $S_{i(\ast)}$ is tenuous.
        \end{enumerate}
    \end{enumerate}

    For all $i<i(\ast)$, we shall see that $p_{\nu_{i(\ast)},\delta,S_{i(\ast)}}^{\ast}\subseteq q_{i}$;
as $\mathrm{tr}(q_{i})\trianglelefteq\nu_{i(\ast)}$, observe that $q_{i}^{[\nu_{i(\ast)}]}\subseteq q_{i}$
by a previous lemma. Also, $q_{i}^{[\nu_{i(\ast)}]}$ and $p_{\nu_{i(\ast)},\delta,S_{i(\ast)}}^{\ast}$
have the same trunk, with the first having a smaller stationary set:
$S_{q_{i}}\subseteq S_{i(\ast)}$, recalling \ref{afterforcingactually}(\ref{whatantichainis2})
we get $p_{\nu_{i(\ast)},\delta,S_{i(\ast)}}^{\ast}\subseteq q_{i}^{[\nu_{i(\ast)}]}$,
thus $p_{\nu_{i(\ast)},\delta,S_{i(\ast)}}^{\ast}\subseteq q_{i}$
and $p_{\nu_{i(\ast)},\delta,S_{i(\ast)}}^{\ast}\subseteq p_{i(\ast)}$.

We need to also see that $p_{i(\ast)}\subseteq p_{\nu_{i(\ast)},\delta,S_{i(\ast)}}^{\ast}$;
assume this doesn't hold. Then, for some $\nu'\in p_{i(\ast)}$, $\nu'\notin p_{\nu_{i(\ast)},\delta,S_{i(\ast)}}^{\ast}$;
let $\delta'$ be the minimal such that $\nu'\restriction\delta'\notin p_{\nu_{i(\ast)},\delta,S_{i(\ast)}}^{\ast}$;
remembering definition \ref{maindef}, necessarily $\nu'\restriction\delta'\in\lim_{\delta'}(p_{\nu_{i(\ast)},\delta',S_{i(\ast)}\cap\delta'}^{\ast})$
and $\nu'\restriction\delta'\in\lim_{\delta'}(r_{\delta'}^{\ast})\setminus(\bigcup\{\lim_{\delta'}(q_{\delta',\eta'}^{\ast}):\eta'\in\Lambda_{\delta'}^{\ast}\})$.
As $\delta'\in S_{i(\ast)}$, there exists $i<i(\ast)$ such that
$\delta'\in S_{q_{i}}$; since for all $\delta''<\delta'$, $\nu'\restriction\delta''\in p_{\nu_{i(\ast)},\delta,S_{i(\ast)}}^{\ast}\subseteq q_{i}$
so $\nu'\restriction\delta'\in\lim_{\delta'}(q_{i})$ and by the construction
of $q_{i}$ as $\nu'\restriction\delta'\in\lim_{\delta'}(r_{\delta'}^{\ast})\setminus(\bigcup\{\lim_{\delta'}(q_{\delta',\eta'}^{\ast}):\eta'\in\Lambda_{\delta'}^{\ast}\})$
it follows that $\nu'\restriction\delta'\notin q_{i}$, a contradiction
to the assumption $\nu'\in p_{i(\ast)}\subseteq q_{i}$.

Finally we have that $p_{\nu_{i(\ast)},\delta,S_{i(\ast)}}^{\ast}=p_{i(\ast)}$,
in addition for all $i<i(\ast)$, $q_{i}\leq_{\mathbb{Q}_{\lambda}}p_{i(\ast)}$
so easily $p_{i(\ast)}$ is the smallest supremum of those conditions.
\end{enumerate}

Finally we can see that player COM has a legal move for each $i<\delta$
thus the forcing $\mathbb{Q}_{\delta}$ is strategically complete
in $\alpha$.
\end{proof}
\end{thm}

By \ref{completenessfornoninaccessibles} and \ref{completenessforinaccessibles}:
\begin{cor}\label{completeness}
For all $\delta\in S_{\ast}\cup\{\lambda\}$,
the forcing $\mathbb{Q}_{\delta}$ is strategically complete in $\mathrm{cf}(\delta)$.
\end{cor}

\begin{thm}\label{chain}
If $\delta\in S_{\ast}\cup\{\lambda\}$, \uline{then} the $\delta^{+}$-chain condition holds for the forcing $\mathbb{Q}_{\delta}$.
\end{thm}

\begin{proof}
Let $\mathcal{A}\subseteq\mathbb{Q}_{\delta}$ be an antichain, then
for all $p,q\in\mathcal{A}$ by Claim \ref{afterforcingactually}(\ref{whatantichainis2}),
$\mathrm{tr}(p)\neq \mathrm{tr}(q)\in\boldsymbol{T}_{\!<\delta}$; recalling the definition
of the good structure $\mathfrak{r}$ it holds that for each $\zeta<\delta$:
$\theta_{\zeta}<\delta$ and as $\delta$ is a strong limit, $\big|\bigcup\limits_{\epsilon<\delta} \prod\limits_{i<\epsilon}\theta_{i} \big| = |\boldsymbol{T}_{\!<\delta}|\leq\delta$;
in particular for any antichain $\mathcal{A}\subseteq\mathbb{Q}_{\delta}$,
$|\mathcal{A}|\leq\delta$.\end{proof}
\begin{cor}
By \ref{completeness} and \ref{chain}, the forcing $\mathbb{Q}_{\lambda}$
is $\leq\lambda$-strategically complete and the $\lambda^{+}$-chain
condition holds for it.\end{cor}
\begin{thm}
If $\lambda$ is an inaccessible cardinal, \uline{then} the forcing
$\mathbb{Q}_{\lambda}$ is $\lambda$- bounding.
\begin{proof}
Let $p_{\ast}\in\mathbb{Q}_{\lambda}$ and $\name{\tau}$ a $\mathbb{Q}_{\lambda}$-
name for a function from $\lambda$ to $\lambda$. We would like to
have a condition $q\geq_{\mathbb{Q}_{\lambda}}p_{\ast}$, $q\in\mathbb{Q}_{\lambda}$
and a function $g:\lambda\rightarrow\lambda$ such that $q\Vdash_{\mathbb{Q}_{\lambda}}``\name{\tau}\leq g"$.
In this proof we denote $\leq$ instead of $\leq_{\mathbb{Q}_{\lambda}}$
when comparing forcing conditions.
\begin{itemize}
\item We will find a sequence $\langle p_{\epsilon},S_{\epsilon},E_{\epsilon},\alpha_{\epsilon}\rangle$
for each $\epsilon<\lambda$ such that:\end{itemize}
\begin{enumerate}
\item \label{bounding1}it holds that $p_{0}=p_{\ast}$,
\item \label{bounding2}$p_{\epsilon}=p_{\varrho,\lambda,S_{\epsilon}}^{\ast}$
for $\varrho=\mathrm{tr}(p_{\ast})$,
\item \label{bounding3}the sequence $\langle p_{\zeta}:\zeta\leq\epsilon\rangle$
is increasing and continuous,
\item \label{bounding4}$E_{\epsilon}$ is a club disjoint to $S_{\epsilon}$,
\item \label{bounding5}the sequence $\langle E_{\epsilon}:\epsilon<\lambda\rangle$
is decreasing,
\item \label{bounding6}for $\epsilon=\zeta+1<\lambda$ it holds that $\alpha_{\epsilon}\in E_{\zeta}$
and $\alpha_{\epsilon}\in S_{\ast}\setminus(S_{\zeta}\setminus(\alpha_{\zeta}+1))$,
\item \label{bounding7}for a limit $\epsilon<\lambda$, $\alpha_{\epsilon}\in E_{\epsilon}$,
\item \label{bounding8}the sequence $\langle\alpha_{\zeta}:\zeta\leq\epsilon\rangle$
will be increasing continuous, consisting of ordinals greater than
$\lg(\varrho)$,
\item \label{bounding9}for $\zeta<\epsilon<\lambda$ it holds that $S_{\zeta}\cap(\alpha_{\zeta}+1)=S_{\epsilon}\cap(\alpha_{\zeta}+1)$,
\item \label{bounding11}for $\epsilon=\zeta+1$, the ordinal $\alpha_{\epsilon}$
represents a level, in which in the corresponding tree the value of
the function in $\zeta$ will be determined, that is:

\begin{enumerate}
\item \label{bounding11(a)}for all $\nu\in p_{\epsilon}\cap\boldsymbol{T}_{\alpha_{\epsilon}}$
it holds that $p_{\epsilon}^{[\nu]}$ forces a value for $\name{\tau}(\zeta)$,
\item \label{bounding11(b)}it holds that $p_{\epsilon}\Vdash_{\mathbb{Q}_{\lambda}}``\name{\tau}(\zeta)\in u_{\zeta}"$
where $u_{\zeta}\subseteq\lambda$ of cardinality $<\lambda$.
\end{enumerate}
\end{enumerate}
\begin{itemize}
\item Next we see that this construction is possible, by induction:

\begin{itemize}
\item For the basis $\epsilon=0$:

We have that $p_{0}=p_{\ast}$, $\alpha_{0}=\lg(\varrho)$ so \ref{bounding1}
holds; $S_{\epsilon}$ is the tenuous set corresponding to $p$ and
let $E_{\epsilon}$ be a club in $\lambda$ disjoint to $S_{\epsilon}$
(as $S_{\epsilon}$ is tenuous).

\item For $\epsilon<\lambda$ limit:

Start with the set $S_{\epsilon}$: let $S_{\epsilon}=\underset{\zeta<\epsilon}{\bigcup}S_{\zeta}\subseteq S_{\ast}$.
Then it is easy to see that clause \ref{bounding9} holds (by the
induction hypothesis); let also $\alpha_{\epsilon}=\underset{\zeta<\epsilon}{\bigcup}\alpha_{\zeta}$
and $E_{\epsilon}=\underset{\zeta<\epsilon}{\bigcap}E_{\zeta}$ ,
observe that $E_{\epsilon}$ is a club disjoint to $S_{\epsilon}$,
so clauses \ref{bounding4} and \ref{bounding5} hold.

Now we will show that the set $S_{\epsilon}$ is indeed tenuous: first,
the set $S_{\epsilon}$ is non-stationary in $\lambda$ as a union
of $\epsilon<\lambda=\mathrm{cf}(\lambda)$ sets that are non-stationary in
$\lambda$ and by Remark \ref{stationarynotreflecting} when $S_{\ast}$
is non-reflecting, $S_{\epsilon}$ is also tenuous, but we have to
prove it in general.

Next, let $\gamma<\lambda$ be an ordinal of uncountable cofinality
and look at $S_{\epsilon}\restriction\gamma$:

If there exists $\zeta<\epsilon$ for which $\gamma<\alpha_{\zeta}$
then as $S_{\epsilon}\cap(\alpha_{\zeta}+1)=S_{\zeta}\cap(\alpha_{\zeta}+1)$
it follows that $S_{\epsilon}\cap\gamma=S_{\zeta}\cap\gamma$ and
since $S_{\zeta}$ is tenuous this set is non-stationary.

For $\gamma=\alpha_{\epsilon}$, first observe that by the definition
of $E_{\epsilon}$ as the limit of the clubs $\langle E_{\zeta}:\zeta<\epsilon\rangle$
and since the sequence of clubs is decreasing, and by \ref{bounding6}
of the induction hypothesis it holds that $\alpha_{\epsilon}\in\underset{\zeta<\epsilon}{\bigcap}E_{\zeta}=E_{\epsilon}$,
this was clause \ref{bounding7}, and so $\alpha_{\epsilon}\notin S_{\epsilon}$.
\begin{itemize}
\item When $\alpha_{\epsilon}$ is regular (and thus inaccessible): by \ref{bounding8}
in the induction hypothesis, the set $\{\alpha_{\zeta}:\mbox{\ensuremath{\zeta}}\mbox{ is a limit ordinal }<\epsilon\}$
is a club of $\alpha_{\epsilon}$, in addition by clause \ref{bounding7}
in the induction hypothesis, for all $\zeta<\epsilon$ limit: $\alpha_{\zeta}\notin S_{\zeta}$
and by clause \ref{bounding9} in the induction hypothesis, for every
$\zeta<\xi<\epsilon$ it holds that $\alpha_{\zeta}\notin S_{\xi}$
and therefore $\alpha_{\zeta}\notin S_{\epsilon}$ and this club is
disjoint to $S_{\epsilon}\restriction\alpha_{\epsilon}$, so this
is not a stationary set.
\item When $\alpha_{\epsilon}$ is singular, the set $S_{\ast}$ doesn't
reflect to $\alpha_{\epsilon}$ by definition, so $S_{\ast}\restriction\alpha_{\epsilon}$
is a non-stationary set, and in particular $S_{\epsilon}\restriction\alpha_{\epsilon}\subseteq S_{\ast}\restriction\alpha_{\epsilon}$
is not a stationary set by \ref{stationarynotreflecting}.
\end{itemize}

Lastly for $\gamma>\alpha_{\epsilon}$:
\begin{itemize}
\item If $\mathrm{cf}(\gamma)>\epsilon$ then for all $\zeta<\epsilon$ it holds
that $S_{\zeta}\restriction\gamma$ is a non-stationary set from clause
\ref{bounding2} of the induction hypothesis, so there is a club of
$\gamma$ disjoint to it, call it $C_{\zeta}$. Letting $C_{\epsilon}=\underset{\zeta<\epsilon}{\bigcap}C_{\zeta}$,
this is a club as the intersection of $\epsilon$ clubs, disjoint
to $S_{\epsilon}$ by its definition, so $S_{\epsilon}\restriction\gamma$
is non-stationary.
\item Otherwise, if $\gamma>\epsilon\geq \mathrm{cf}(\gamma)$ in particular it follows
that $\gamma$ is singular, thus $S_{\ast}$ doesn't reflect to $\gamma$
and so also $S_{\epsilon}\subseteq S_{\ast}$ using Remark \ref{stationarynotreflecting}.
\end{itemize}

Let $p_{\epsilon}=p_{\varrho,\lambda,S_{\epsilon}}^{\ast}$ so clauses
\ref{bounding2} hold. Moreover, $p_{\epsilon}\subseteq\underset{\zeta<\epsilon}{\bigcap}p_{\varrho,\lambda,S_{\zeta}}^{\ast}$,
why? assume there is some $\nu'\in\underset{\zeta<\epsilon}{\bigcap}p_{\varrho,\lambda,S_{\zeta}}^{\ast}\setminus p_{\epsilon}$,
as $\varrho\trianglelefteq\nu'$ there is some minimal $\delta'$
for which $\nu'\restriction\delta'\notin p_{\epsilon}$; then $\nu'\restriction\delta'\in\lim_{\delta'}(p_{\epsilon}\cap\boldsymbol{T}_{<\delta'})$
and by definition \ref{maindef} neccessarily $\nu'\restriction\delta'\in\lim_{\delta'}(r_{\delta'}^{\ast})\setminus(\bigcup\{\lim_{\delta'}(q_{\delta',\eta'}^{\ast}):\eta'\in\Lambda_{\delta'}^{\ast}\})$.
Since for some $\zeta<\epsilon$, $\delta'\in S_{\zeta}$, it follows
that $\nu'\restriction\delta'\notin p_{\varrho,\lambda,S_{\zeta}}^{\ast}$
and so $\nu'\notin p_{\varrho,\lambda,S_{\zeta}}^{\ast}$, a contradiction.
Thus, $p_{\epsilon}=\underset{\zeta<\epsilon}{\bigcap}p_{\varrho,\lambda,S_{\zeta}}^{\ast}$
and \ref{bounding3} hold.

\item For $\epsilon=\zeta+1$:

This is the main case, as here we deal with clause \ref{bounding11}
that is responsible for determining the values of the function.

Define the following set:
\[
\mathscr{J}_{\epsilon}=\{r\in\mathbb{Q}_{\lambda}:r\mbox{ forces a value on }\name{\tau}(\zeta)\wedge p_{\zeta}\leq_{\mathbb{Q}_{\lambda}}r\wedge\lg(\mathrm{tr}(r))>\alpha_{\zeta}\}
\]
and observe:

(a) This set is dense above $p_{\zeta}$: for all $p\in\mathbb{Q}_{\lambda}$
with $p_{\zeta}\leq p$, we will find a condition $r$ stronger than
$p$ that forces a value on $\name{\tau}(\zeta)$ and if $\lg(\mathrm{tr}(r))>\alpha_{\zeta}$
doesn't hold, we can extend $r$ to a stronger condition with long
enough trunk.

(b) The set is open: for all $q\in\mathscr{J}_{\epsilon}$ and $r\geq q$,
$q$ forces a value on $\name{\tau}(\zeta)$ and therefore, so does
$r$, $\lg(\mathrm{tr}(r))\geq\lg(\mathrm{tr}(q))>\alpha_{\zeta}$ and of course that
$p_{\zeta}\leq q\leq r$ .

Now define a set $\Lambda_{\epsilon}=\{\mathrm{tr}(r):r\in\mathscr{J}_{\epsilon}\}$
and for every $\eta\in\Lambda_{\epsilon}$ choose some $q_{\epsilon,\eta}\in\{r\in\mathscr{J}_{\epsilon}:\mathrm{tr}(r)=\eta\}$.

Choose a set $\Lambda_{\epsilon}^{1}\subseteq\Lambda_{\epsilon}$
that is maximal under the restriction that for any different $\eta,\nu\in\Lambda_{\epsilon}^{1}$,
$\nu\notin q_{\epsilon,\eta}\vee\eta\notin q_{\epsilon,\nu}$ ; let
$\bar{q}_{\epsilon}=\langle q_{\epsilon,\eta}:\eta\in\Lambda_{\epsilon}^{1}\rangle$.
\begin{itemize}
\item Observe that the sequence $\bar{q}_{\epsilon}=\langle q_{\epsilon,\eta}:\eta\in\Lambda_{\epsilon}^{1}\rangle\in\Xi_{\lambda}$
because:

(1) $\Lambda_{\epsilon}^{1}\subseteq\boldsymbol{T}_{<\lambda}$,

(2) for all $\eta\in\Lambda_{\epsilon}^{1}$ it holds that $q_{\epsilon,\eta}\in\mathbb{Q}_{\lambda}\subseteq\mathbb{Q}_{\lambda}^{0}$
and $\mathrm{tr}(q_{\epsilon,\eta})=\eta$,

(3) if $\eta,\nu\in\Lambda_{\epsilon}^{1}$ are different, then by
the definition of $\Lambda_{\epsilon}^{1}$ it holds that $\mathrm{tr}(q_{\epsilon,\nu})=\nu\notin q_{\epsilon,\eta}\vee \mathrm{tr}(q_{\epsilon,\eta})=\eta\notin q_{\epsilon,\nu}$,

(4) \label{(4)r=00003Dp}it holds that $r_{\bar{q}_{\epsilon}}^{\ast}=\{\rho\in\boldsymbol{T}_{<\lambda}:(\exists\eta\in\Lambda_{\epsilon}^{1})(\rho\in q_{\epsilon,\eta})\}=p_{\zeta}$,
in particular it belongs to $\mathbb{Q}_{\lambda}\subseteq\mathbb{Q}_{\lambda}^{0}$;
observe that for all $\eta\in\Lambda_{\epsilon}^{1}$, it holds that
$q_{\epsilon,\eta}\subseteq p_{\zeta}$ and so $r_{\bar{q}_{\epsilon}}^{\ast}\subseteq p_{\zeta}$.
Assume via contradiction that $\nu\in p_{\zeta}\setminus r_{\bar{q}_{\epsilon}}^{\ast}$
then there is $p_{\zeta}^{[\nu]}\leq_{\mathbb{Q}_{\lambda}}q$ that
forces a value for $\name{\tau}(\zeta)$ and its trunk is longer
than $\alpha_{\zeta}$, so $q\in\mathscr{J}_{\epsilon}$ and $\mathrm{tr}(q)\in\Lambda_{\epsilon}$.
If $\mathrm{tr}(q)\in\Lambda_{\epsilon}^{1}$ we get $\mathrm{tr}(q)\in r_{\bar{q}_{\epsilon}}^{\ast}$,
a contradiction to the assumption; hence there is $\nu'\in\Lambda_{\epsilon}^{1}$
such that $\mathrm{tr}(q)\in q_{\epsilon,\nu'}\vee \mathrm{tr}(q_{\epsilon,\nu'})\in q$
so again we get $\mathrm{tr}(q)\in r_{\bar{q}_{\epsilon}}^{\ast}$ but the
later conjunct contradict the choice of $\nu$ and $\nu\in r_{\bar{q}_{\epsilon}}^{\ast}$,
a contradiction.

\end{itemize}

For all $\eta\in\Lambda_{\epsilon}^{1}$ it holds that $q_{\epsilon,\eta}$
forces a value on $\name{\tau}(\zeta)$; call this value $\gamma_{\epsilon,\eta}$.
In addition let $C_{\eta}$ be a club disjoint to $S_{q_{\epsilon,\eta}}$.

\textbf{First, define an approximation for the club $E_{\epsilon}$.}

$E_{\epsilon}'=\{\delta\in E_{\zeta}:\delta>\alpha_{\zeta}\mbox{ is a limit ordinal such that }\nu'\in\Lambda_{\epsilon}^{1}\cap\boldsymbol{T}_{\!<\delta}\rightarrow\delta\in C_{\nu'}\mbox{ and }\nu\in p_{\zeta}\cap\boldsymbol{T}_{\!<\delta}\rightarrow\nu\in q_{\epsilon,\eta}\mbox{ for some }\eta\in\boldsymbol{T}_{\!<\delta}\cap\Lambda_{\epsilon}^{1})\}$

The set $E_{\epsilon}'$ is a club in $\lambda$:
\begin{itemize}
\item \uline{Closed}- for every increasing sequence of ordinals $\langle\delta_{i}:i<\zeta^{\ast}\rangle$
such that for all $i<\zeta^{\ast}$: $\delta_{i}\in E_{\epsilon}'$
and $\zeta^{\ast}<\lambda$, their limit $\delta=\underset{i<\zeta^{\ast}}{\lim}\delta_{i}$
is of course a limit ordinal and belongs to $E_{\zeta}$. In addition,
for all $\nu'\in\Lambda_{\epsilon}^{1}$ with $\lg(\nu')<\delta$
there is $j_{0}<\zeta^{\ast}$ such that for all $j_{0}<j<\zeta^{\ast}$
it holds that $\lg(\nu')<\delta_{j}$ (as $\delta$ is defined to
be the limit of those), then $\delta_{j}\in C_{\nu'}$ and since $C_{\nu'}$
is a club it follows that $\delta\in C_{\nu'}$, as the limit of $\langle\delta_{j}:j_{0}<j<\zeta^{\ast}\rangle$.

Lastly if $\nu\in p_{\zeta}\cap\boldsymbol{T}_{\!<\delta}$ then $\lg(\nu)<\delta$
hence for some $i<\zeta^{\ast}$, $\lg(\nu)<\delta_{i}$ hence $\nu\in\boldsymbol{T}_{<\delta_{i}}$
hence $\nu\in p_{\zeta}\cap\boldsymbol{T}_{\!<\delta}$. As $\delta_{i}\in E_{\epsilon}'$
necessarily there is $\eta\in\boldsymbol{T}_{<\delta_{i}}\cap\Lambda_{\epsilon}^{1}$
such that $\nu\in q_{\epsilon,\eta}$ but clearly $\eta\in\boldsymbol{T}_{\!<\delta}\cap\Lambda_{\epsilon}^{1}$
so we are done.

\item \uline{Unbounded}-\textbf{\textit{ }}otherwise, the set $E_{\epsilon}'$
was bounded by some $\xi<\lambda$; then for every limit $\xi<\delta\in E_{\zeta}$,
$\delta\notin E_{\epsilon}'$ so \uline{either} \textit{(1)} $\exists\nu'\in\Lambda_{\epsilon}^{1}\cap\boldsymbol{T}_{\!<\delta}$
such that $\delta\notin C_{\nu'}$ \uline{o}\textit{\uline{r}}\textit{
(2) $(\exists\nu\in p_{\zeta}\cap\boldsymbol{T}_{\!<\delta})(\forall\eta\in\boldsymbol{T}_{\!<\delta}\cap\Lambda_{\epsilon}^{1})(\nu\notin q_{\epsilon,\eta})$}.
As $E_{\zeta}\setminus(\xi+1)$ is stationary, for some $W\subseteq E_{\zeta}\setminus(\xi+1)$
stationary in $\lambda$, for all $\delta\in W$ the same case occurs.
If it is by case \textit{(2)}, as $|\boldsymbol{T}_{<\alpha}|<\lambda$
for $\alpha<\lambda$, by Fodor's lemma there is a stationary set
$W_{2}\subseteq E_{\zeta}\setminus(\xi+1)$ such that for all $\delta\in W_{2}$
we can choose the same $\nu\in p_{\zeta}\cap\boldsymbol{T}_{\!<\delta}$-
a contradiction. Thus \textit{(2)} is impossible and if it is by case\textit{
(1)} so $$\delta \notin \bigcap\limits_{\nu'\in\Lambda_{\epsilon}^{1}\cap\boldsymbol{T}_{\!<\delta}} C_{\nu'} \subseteq \bigcap\limits_{\nu'\in\Lambda_{\epsilon}^{1}\cap\boldsymbol{T}_{\!<\xi}} C_{\nu'}=C$$
as $|\Lambda_{\epsilon}^{1}\cap\boldsymbol{T}_{\!<\delta}|<\lambda$.
For any $\xi<\delta\in E_{\zeta}$ we get 
$$\{\delta\in(\xi,\lambda)\cap E_{\zeta}:\delta\mbox{ is a limit ordinal}\}\cap C=\varnothing;$$
however, $C$ is a club as it is the intersection of less than $\lambda$
clubs --- a contradiction.
\end{itemize}

\textbf{Define the level.}

We would like to have an ordinal $\delta$ for which the following
properties hold:
\begin{enumerate}
\item [(a)]\setcounter{enumi}{0}\label{level1}$\delta\in E_{\epsilon}'\cap S_{\ast}$,
\item [(b)]\setcounter{enumi}{1}\label{level2}$\alpha_{\zeta}<\delta$
(follows from (a)),
\item [(c)]\setcounter{enumi}{2}\label{level3}$r_{\delta}^{\ast}=p_{\zeta}\cap\boldsymbol{T}_{\!<\delta}$,
\item [(d)]\setcounter{enumi}{3}\label{level4}It holds that $\bar{q}_{\delta}^{\ast}=\langle q_{\epsilon,\eta}\cap\boldsymbol{T}_{\!<\delta}:\eta\in\Lambda_{\epsilon}^{1}\cap\boldsymbol{T}_{\!<\delta}\rangle$.
\end{enumerate}

An ordinal with those properties exists:

First, by Claim \ref{useanti} there is a stationary set of $\delta\in S_{\ast}$
such that clause d. holds for and call it $S^{+}$; as $E_{\epsilon}'$
is a club, we get that $S^{+}\cap E_{\epsilon}'$ is stationary. Observe
that for all $\delta\in S^{+}\cap E_{\epsilon}'$ from clause d. it
follows that $r_{\delta}^{\ast} = \bigcup\limits_{\eta\in\Lambda_{\delta}^{\ast}}q_{\eta}^{\ast} = \bigcup\limits_{\nu\in\Lambda_{\epsilon}^{1}\cap\boldsymbol{T}_{\!<\delta}} q_{\epsilon,\nu}\cap\boldsymbol{T}_{\!<\delta}$,
in addition by the definition of $E_{\epsilon}'$ it holds that $p_{\zeta}\cap\boldsymbol{T}_{\!<\delta}=\underset{\nu\in\Lambda_{\epsilon}^{1}\cap\boldsymbol{T}_{\!<\delta}}{\bigcup}q_{\epsilon,\nu}\cap\boldsymbol{T}_{\!<\delta}$,
so for all $\delta\in S^{+}\cap E_{\epsilon}'$ clause c. holds, as
this set is not empty (as a stationary set) there is such $\delta$,
and we are done.

\uline{Let $\alpha_{\epsilon}=\delta$.} Observe that in particular
it follow that $\Lambda_{\epsilon}^{1}\cap\boldsymbol{T}_{<\alpha_{\epsilon}}=\Lambda_{\alpha_{\epsilon}}^{\ast}$.

We can now let $E_{\epsilon}=E_{\epsilon}'\setminus(\alpha_{\epsilon}+1)$,
notice that also $E_{\epsilon}$ is a club in $\lambda$.

\textbf{Define the tenuous set of $p_{\epsilon}$.}

First in the $\alpha_{\epsilon}$-th level we define the set of all
the limits formed from the conditions of $\bar{q}_{\alpha_{\epsilon}}^{\ast}$:
\[
\Lambda_{\epsilon}^{2} = p_{\zeta} \cap \boldsymbol{T}_{\alpha_{\epsilon}} \cap \Big( \bigcup\{\lim(q_{\alpha_{\epsilon},\nu}^{\ast}):\nu\in\Lambda_{\alpha_{\epsilon}}^{\ast}\} \Big)
\]
For $\eta\in\Lambda_{\epsilon}^{2}$, by the definition above and
the definition of the level there is unique $\nu\in\Lambda_{\alpha_{\epsilon}}^{\ast}$
with $\eta\in\lim(q_{\alpha_{\epsilon},\nu}^{\ast})$, as $q_{\epsilon,\nu}\cap\boldsymbol{T}_{<\alpha_{\epsilon}}=q_{\alpha_{\epsilon},\nu}^{\ast}$
and recalling definition \ref{maindef}, the fact that $\eta\in\lim(q_{\alpha_{\epsilon},\nu}^{\ast})$
also implies $\eta\in q_{\epsilon,\nu}$; let $r_{\eta}:=(q_{\epsilon,\nu})^{[\eta]}$.

Now, define $S_{\epsilon}^{1}=\cup\{S_{r_{\eta}}\setminus(\alpha_{\epsilon}+1):\eta\in\Lambda_{\epsilon}^{2}\}$.
Observe that for every $\eta\in\Lambda_{\epsilon}^{2}$, $S_{r_{\eta}}\subseteq S_{q_{\epsilon,\nu}}$
for some $\nu\in\Lambda_{\alpha_{\epsilon}}^{\ast}\subseteq\boldsymbol{T}_{<\alpha_{\epsilon}}$
(follows from $r_{\eta}=(q_{\epsilon,\nu})^{[\eta]}$ and \ref{afterforcingactually}.(\ref{enu:afterforcingactually2})).
Thus $S_{\epsilon}^{1}\subseteq\cup\{S_{q_{\epsilon,\nu}}:\nu\in\Lambda_{\alpha_{\epsilon}}^{\ast}\}$
and this is a union of $\leq|\boldsymbol{T}_{<\alpha_{\epsilon}}|\leq\alpha_{\epsilon}$
sets, each one is a tenuous subset of $S_{\ast}\setminus(\alpha_{\epsilon}+1)$
and in particular non-stationary in $\lambda$. So their union will
be the union of $\leq\alpha_{\epsilon}<\lambda$ (as $\lambda$ is
inaccessible) non-stationary sets, and as $\lambda=\mathrm{cf}(\lambda)$ and
by Claim \ref{stationaryunions} it follows that $S_{\epsilon}^{1}$
is a non-stationary subset of $\lambda$.

Next, let $\alpha_{\epsilon}<\delta<\lambda$:
\begin{itemize}
\item If $\delta$ is an inaccessible cardinal in $S_{\ast}$, we want to
show that $S_{\epsilon}^{1}\restriction\delta$ is non-stationary
in $\delta$: as $2^{\alpha_{\epsilon}}<\delta$ (by inaccessibility
of $\delta$) and since for all $\eta\in\Lambda_{\epsilon}^{2}$ the
set $S_{r_{\eta}}$ is tenuous, in particular $S_{r_{\eta}}\restriction\delta$
is non-stationary so $S_{\epsilon}^{1}$ is the union of $<\delta=\mathrm{cf}(\delta)$
non-stationary set and by Claim \ref{stationaryunions} it is not
stationary.
\item Else, in particular $S_{\ast}$ does not reflect to $\delta$, then
the set $S_{\ast}\restriction\delta$ is non-stationary in $\delta$
and so also in $S_{\epsilon}^{1}\restriction\delta$ by \ref{stationarynotreflecting}.
\end{itemize}

This shows that $S_{\epsilon}^{1}$ is tenuous and therefore $S_{\epsilon}=S_{\zeta}\cup\{\alpha_{\epsilon}\}\cup S_{\epsilon}^{1}$
that is also tenuous.

Moreover we can see that $E_{\epsilon}$ is disjoint to $S_{\zeta}\cup\{\alpha_{\epsilon}\}$
as a subset of $E_{\zeta}\setminus(\alpha_{\epsilon}+1)$ and by the
induction hypothesis; in addition for all $\delta\in E_{\epsilon}$,
$\delta\in\underset{\nu'\in\Lambda_{\epsilon}^{1}\cap\boldsymbol{T}_{\!<\delta}}{\bigcap}C_{\nu'}$.
For all $\eta\in\Lambda_{\epsilon}^{2}$ it holds that $S_{r_{\eta}}\subseteq S_{q_{\epsilon,\nu}}$
for some $\nu\in\Lambda_{\epsilon}^{1}\cap\boldsymbol{T}_{<\alpha_{\epsilon}}\subseteq\Lambda_{\epsilon}^{1}\cap\boldsymbol{T}_{\!<\delta}$,
so the set $C_{\nu}$ is disjoint to $S_{r_{\eta}}$ and in particular
$\delta\notin S_{r_{\eta}}$. Finally we have that $S_{\epsilon}\cap E_{\epsilon}=\varnothing$.

\textbf{Define the condition.}

The condition will be $p_{\epsilon}=p_{\varrho,\lambda,S_{\epsilon}}^{\ast}$
and so $p_{\epsilon}\in\mathbb{Q}_{\lambda};$ we would like $p_{\epsilon}\subseteq p_{\zeta}$
to hold, for the condition to be stronger than in the previous level;
this is formed as we are using a larger tenuous set than the one of
$p_{\zeta}$.

\uline{Statement:} For all $\rho\in p_{\zeta}$, $\rho\in p_{\epsilon}$
if and only if ($\lg(\rho)<\alpha_{\epsilon}$) or ($\alpha_{\epsilon}\leq\lg(\rho)$
and $(\forall\eta\in\Lambda_{\epsilon}^{2})(\rho\in r_{\eta})$).

\uline{Proof:}
\begin{enumerate}
\item If $\rho\in p_{\epsilon}$ then either (a) $\lg(\rho)<\alpha_{\epsilon}$
or (b) $\alpha_{\epsilon}\leq\lg(\rho)$. In (b), let $\eta\in\Lambda_{\epsilon}^{2}$
and supposse $\delta_{1}$ is minimal such that $\rho\restriction\delta_{1}\notin r_{\eta}$
so $\delta_{1}\in S_{r_{\eta}}$, in which case $\delta_{1}$ is successful
and $\rho\restriction\delta_{1}\in\lim_{\delta_{1}}(r_{\delta_{1}}^{\ast})\setminus(\bigcup\{\lim_{\delta_{1}}(q_{\delta_{1},\eta'}^{\ast}):\eta'\in\Lambda_{\delta_{1}}^{\ast}\})$.
Thus $\rho\restriction\delta_{1}\notin p_{\epsilon}\Rightarrow\rho\notin p_{\epsilon}$-
a contradiction.

It holds then that $\alpha_{\epsilon}\leq\lg(\rho)\rightarrow(\forall\eta\in\Lambda_{\epsilon}^{2})(\rho\in r_{\eta})$.

\item For the other direction, if $\rho$ is such that $\lg(\rho)<\alpha_{\epsilon}$,
$\rho\in p_{\zeta}$ and if $\alpha_{\epsilon}\leq\lg(\rho)$, let
$\rho\restriction\alpha_{\epsilon}=:\eta$ then $\eta\in\Lambda_{\epsilon}^{2}\wedge\rho\in r_{\eta}$.
If $\rho\notin p_{\epsilon}$, for some $\lg(\varrho)<\delta_{1}\in S_{\epsilon}$,
$$\rho\restriction\delta_{1}\in\lim_{\delta_{1}}(r_{\delta_{1}}^{\ast})\setminus \Big(\bigcup\{\lim_{\delta_{1}}(q_{\delta_{1},\eta'}^{\ast}):\eta'\in\Lambda_{\delta_{1}}^{\ast}\} \Big).$$

\begin{enumerate}
\item If $\delta_{1}<\alpha_{\epsilon}$ then $\delta_{1}\in S_{\zeta}$
and $\rho\restriction\delta_{1}\notin p_{\zeta}$- a contradiction.
\item If $\delta_{1}>\alpha_{\epsilon}$ then $\delta_{1}\in S_{\epsilon}^{1}$
so for some $\eta'\in\Lambda_{\epsilon}^{2}$, $\delta_{1}\in S_{r_{\eta'}}$.
and $\rho\restriction\delta_{1}\notin r_{\eta'}$- a contradiction.
\item If $\delta_{1}=\alpha_{\epsilon}$ we have $\rho\in r_{\eta}=(q_{\alpha_{\epsilon},\nu}^{\ast})^{[\eta]}$-
a contradiction.
\end{enumerate}

We are done with the statement.

Now observe:
\begin{itemize}
\item We can easily verify that $p_{\zeta}\leq_{\mathbb{Q}_{\lambda}}p_{\epsilon}$.
\item The set $\{r_{\eta}:\eta\in\Lambda_{\epsilon}^{2}\}$ is predense
above $p_{\epsilon}$ in $\mathbb{Q}_{\lambda}$: let $p_{\epsilon}\leq q$,
assume there are no forcing conditions in $\{q\cap r_{\eta}:\eta\in\Lambda_{\epsilon}^{2}\}$;
recall that $\rho\in p_{\epsilon}\Leftrightarrow\rho\in\bigcap\{r_{\eta}:\eta\in\Lambda_{\epsilon}^{2}\}$,
then $q=q\cap p_{\epsilon}=\bigcap\{q\cap r_{\eta}:\eta\in\Lambda_{\epsilon}^{2}\}$-
a contradiction, as the right side cannot be a condition.
\item as in fact the pruning had been to get $p_{\epsilon}$ exactly by
this set.
\item Thus, as for all $\eta\in\Lambda_{\epsilon}^{2}$ it holds that $r_{\eta}\Vdash\name{\tau}(\zeta)=\gamma_{\epsilon,\nu_{\eta}}$
for some $\nu_{\eta}\trianglelefteq\eta$, we can write $u_{\zeta}=\{\gamma_{\epsilon,\nu_{\eta}}:\eta\in\Lambda_{\epsilon}^{2}\}$
and have $p_{\epsilon}\Vdash"\name{\tau}(\zeta)\in u_{\zeta}"$.
\end{itemize}

Clause \ref{bounding11} holds and so the construction is possible.

\end{enumerate}
\item Let $S'=\underset{\epsilon<\lambda}{\bigcup}S_{\epsilon}$, this is
non-stationary because $\underset{\epsilon<\lambda}{\Delta}E_{\epsilon}\cap S'=\varnothing$
and a tenuous set as for all $\delta<\lambda$ there is $\epsilon<\lambda$
with $S'\cap\delta=S_{\epsilon}\cap\delta$ (by clause \ref{bounding9}).
\item Lastly, let $q=p_{\varrho,\lambda,S'}^{\ast}$ then indeed $p\leq q$
and we can define $g:\lambda\rightarrow\lambda$ by: for $\epsilon<\lambda$
let $g(\epsilon)=\sup\{u_{\epsilon}\}$ where $u_{\epsilon}$ is from
clause (\ref{bounding11(b)}) in our induction, so as $u_{\zeta}$
is a subset of $\lambda$ of cardinality $<\lambda$, clearly $g(\epsilon)<\lambda$
indeed. So $g$ is a function from $\lambda$ into $\lambda$ which
belongs to $\boldsymbol{V}$. Also by clause (\ref{bounding11(b)})
we have $p_{\epsilon+1}\Vdash``\name{\tau}(\epsilon)\in u_{\epsilon}"$
hence $p_{\epsilon+1}\Vdash``\name{\tau}(\epsilon)\leq g(\epsilon)"$.
But $q$ is above $p_{\epsilon+1}$ for every $\epsilon<\lambda$
hence $q\Vdash``\name{\tau}(\epsilon)\leq g(\epsilon)"$.
\end{itemize}

As $p$ is stronger then our original $p$ we are done proving the
theorem.

\end{itemize}
\end{proof}
\end{thm}
\begin{cor}
The forcing $\mathbb{Q}_{\lambda}$ resembles Random Real forcing
for $\lambda$.
\end{cor}
\selectlanguage{british}%
\begin{onehalfspace}
{\large{}\newpage{}}{\large \par}
\end{onehalfspace}
\selectlanguage{english}%

\end{document}